\documentclass{amsart}
\usepackage{amsmath, amsthm, amssymb}
\usepackage{amscd}

\allowdisplaybreaks

\usepackage{hyperref}

\newtheoremstyle{citing}
{}{}
{\itshape}
{}{\bfseries}
{.}
{ }
{\thmnote{#3}}
\theoremstyle{plain}
\newtheorem{thm}{Theorem}[section]
\newtheorem{corol}[thm]{Corollary}
\newtheorem{lemma}[thm]{Lemma}
\newtheorem{prop}[thm]{Proposition}
\newtheorem{example}[thm]{Example}

\numberwithin{equation}{section}

\theoremstyle{definition}
\newtheorem{defn}[thm]{Definition}
\newtheorem{remark}[thm]{Remark}

\theoremstyle{citing}

\newcommand{\inv}[1]{\frac{1}{ #1 }}

\newcommand{\maxsym}{\vee}
\newcommand{\minsym}{\wedge}
\newcommand{\F}{\mathcal{F}}
\newcommand{\G}{\mathcal{G}}
\renewcommand{\P}{\mathbb{P}}
\newcommand{\E}{\mathbb{E}}
\newcommand{\eps}{\varepsilon}
\renewcommand{\O}{\Omega}
\newcommand{\one}{\textbf{1}}
\newcommand{\umd}{\textsc{umd} }
\newcommand{\R}{\mathbb{R}}
\newcommand{\N}{\mathbb{N}}
\newcommand{\n}{\Vert}

\begin{document}

\author{Sonja Cox}
\author{Mark Veraar}

\address{Delft Institute of Applied Mathematics\\
Delft University of Technology \\ P.O. Box 5031\\ 2600 GA Delft\\The
Netherlands} \email{S.G.Cox@tudelft.nl, M.C.Veraar@tudelft.nl}

\title[Decoupling]{Vector-valued decoupling and the Burkholder-Davis-Gundy inequality}



\thanks{The authors wish to thank an anonymous referee for his/her thorough reading of the manuscript, and for the many useful comments which improved the results and the presentation of the paper. The second named author was originally supported by the VICI subsidy 639.033.604 and later on by the VENI subsidy 639.031.930 of the Netherlands Organisation for Scientific Research (NWO)}

\date\today

\begin{abstract}
Let $X$ be a \mbox{(quasi-)}Banach space. Let $d=(d_n)_{n\geq 1}$ be an $X$-valued sequence of random variables adapted to a filtration $(\mathcal{F}_n)_{n\geq 1}$ on a probability space $(\Omega,\mathcal{A},\P)$, define $\mathcal{F}_{\infty}:=\sigma(\mathcal{F}_n:n\geq 1)$ and let $e=(e_n)_{n\geq 1}$ be a $\mathcal{F}_{\infty}$-conditionally independent sequence on $(\Omega,\mathcal{A},\P)$ such that $\mathcal{L}(d_n\,|\,\F_{n-1}) = \mathcal{L}(e_n\,|\,\mathcal{F}_{\infty})$ for all $n\geq 1$ ($\F_0=\{\O,\varnothing\}$). If there exists a $p\in
(0,\infty)$ and a constant $D_p$ independent of $d$ and $e$ such that one has,
for all $n\geq 1$,
\begin{align}\label{abs:decoup}
\mathbb{E} \Big\n \sum_{k=1}^{n} d_k \Big\n^p & \leq D_p^p \mathbb{E} \Big\n \sum_{k=1}^{n} e_k
\Big\n^p, \tag{*}
\end{align}
then $X$ is said to \emph{satisfy the decoupling inequality for $p$}. It has
been proven that $X$ is a \textsc{umd} space if and only if both
\eqref{abs:decoup} and the reverse estimate hold for some (all)
$p\in(1,\infty)$. However, in earlier work we proved that the space $L^1$,
which is not a \textsc{umd} space, satisfies the decoupling inequality for all
$p\geq 1$.

Here we prove that if the decoupling inequality is satisfied in $X$ for some
$p\in(0,\infty)$ then it is satisfied for all $p\in (0,\infty)$. We consider
the behavior of the constant $D_p$ in \eqref{abs:decoup}. We examine its
relation to the norm of the Hilbert transform on $L^p(X)$ and showing that if
$X$ is a Hilbert space then there exists a universal constant $D$ such that \eqref{abs:decoup} holds with $D_p=D$, for all $p\in [1,\infty)$.

An important motivation to study decoupling inequalities is that they play a
key role in the recently developed theory for stochastic integration in Banach
spaces. We extend the available theory, proving a $p^\textrm{th}$-moment
Burkholder-Davis-Gundy inequality for the stochastic integral of an $X$-valued
process, where $X$ is a \textsc{umd} space and $p\in(0,\infty)$.
\end{abstract}

\maketitle

\section{Introduction\label{sec:intro}}

In this article we study a decoupling inequality for $X$-valued tangent sequences, where $X$ is a \mbox{(quasi-)}Banach space  (see Section \ref{s:pre}). Our motivation lies in the role this inequality plays in the development of theory for stochastic integration in Banach spaces \cite{mcconnell:decoup}, \cite{vanNeervenVeraarWeis}; we shall elaborate on this below. However, the decoupling inequality has attracted attention in its own right, see \cite{delaPG}, \cite{KwWo} and references therein. Let us begin with a formal definition of the decoupling inequality.\par
Let $X$ be a \mbox{(quasi-)}Banach space. Let $(\O, (\mathcal{F}_n)_{n\geq 1},{\mathcal A}, \P)$ be a complete probability space and let $(d_n)_{n\geq 1}$ be an $(\mathcal{F}_n)_{n\geq 1}$-adapted sequence of $X$-valued random variables. We adopt the convention that $\F_0=\{\Omega,\varnothing\}$. Set $\F_{\infty}:=\sigma(\F_n:n\geq 1)$. A \emph{$\F_{\infty}$-decoupled tangent sequence of} $(d_n)_{n\geq 1}$ is a sequence $(e_n)_{n\geq 1}$ of $X$-valued random variables on $(\O, {\mathcal A}, \P)$ satisfying two properties. Firstly, we assume that for any $B\in \mathcal{B}(X)$ (the Borel-measurable sets of $X$) we have:
\[\begin{aligned}
\P(d_n\in B\,|\,\mathcal{F}_{n-1})=\P(e_n\in B\,|\,\F_{\infty}),
\end{aligned}\]
and secondly we assume that $(e_n)_{n\geq 1}$ is $\F_{\infty}$-conditionally independent, i.e.\ for
every $n\geq 1$ and every $B_1,\ldots,B_n\in\mathcal{B}(X)$ we
have:
\[\begin{aligned}
\P(e_1\in B_1,\ldots, e_n\in B_n\,|\,\F_{\infty})=\P(e_1\in
B_1\,|\,\F_{\infty})\cdot \ldots \cdot \P(e_n\in B_n\,|\,\F_{\infty}).
\end{aligned}\]
Here $\P(C\,|\,\mathcal{F}_{n-1}) = \E(\one_{C}\,|\,\mathcal{F}_{n-1})$ if $C\in \mathcal{A}$. We wish to emphasize that the definition of a decoupled tangent sequence depends on a filtration and on a sequence adapted to that filtration. However, in what follows we shall omit the reference to the $\sigma$-algebra if it is clear from the context.\par
Kwapie{\'n} and Woyczy{\'n}ski introduced the concept of decoupled tangent sequences in
\cite{KwWo1}. For details on the subject we refer to the monographs
\cite{delaPG,KwWo} and the references therein. It is shown there that given a sequence $(d_n)_{n \geq 1}$ of $(\F_n)_{n \geq 1}$-adapted random
variables on $(\Omega,\mathcal{A},\P)$ one can, by an extension of the probability space, construct a decoupled tangent sequence of $(d_n)_{n\geq 1}$. One easily checks that any two $\F_{\infty}$-decoupled tangent sequences of a $(\F_n)_{n\geq 1}$-adapted sequence $(d_n)_{n \geq 1}$ share the same law.\par
We recall the following basic example (see also Lemma \ref{lem:preddcoupl} below; and \cite[Section 4.3]{KwWo} and
\cite[Chapter 6]{delaPG} where many more examples can be found).
\begin{example}\label{e:decoup}
Let $d_n = \xi_n v_n$, where $(\xi_n)_{n\geq 1}$ is a sequence of independent random
variables with values in $\R$ and $(v_n)_{n\geq 1}$ is an $X$-valued $(\F_n)_{n\geq
0}$-predictable sequence; $\F_n := \sigma(\xi_1, \ldots, \xi_n)$ for $n\geq 1$ and $\F_0 := \{\O, \varnothing\}$. Let $(\tilde{\xi}_n)_{n\geq 1}$ be an independent copy of $(\xi_n)_{n\geq 1}$,
then a decoupled tangent sequence of $d_n$ is given by $e_n = \tilde{\xi}_n v_n$ for $n\geq 1$.
\end{example}
Let $(\xi_{n})_{n\geq 1}$ be a sequence (of $X$-valued random variables). The \emph{difference sequence of} $(\xi_n)_{n\geq 1}$ is the sequence $(\xi_n-\xi_{n-1})_{n\geq 1}$, with the understanding that $\xi_0\equiv 0$.
\begin{defn}\label{def:dec}
Let $X$ be a \mbox{(quasi-)}Banach space and let $p\in (0,\infty)$. We say that \emph{the decoupling
inequality holds in $X$ for $p$} if there exists a constant $D_p$ such that for all complete probability spaces $(\O,\mathcal{A},(\F_{n})_{n\geq 1},\P)$ and every $X$-valued $(\F_n)_{n\geq1}$-adapted $L^p$-sequence $f$ (i.e.\ $f=(f_n)_{n\geq 1}\subseteq L^p(\Omega,X)$):
\begin{align}\label{decoup}
\Vert f_n \Vert_{p} \leq D_p \Vert g_n \Vert_{p},
\end{align}
for every $n\geq 1$, where $g$ is a sequence whose difference sequence is an $\mathcal{F}_{\infty}$-decoupled tangent sequence
of the difference sequence of $f$. The least constant $D_p$ for which \eqref{decoup} holds is denoted by $D_p(X)$.
\end{defn}
We will refer to a sequence $g=(g_n)_{n\geq 1}$ whose difference sequence is a $\F_{\infty}$-decoupled tangent sequence
of the difference sequence of $f$, where $f$ is adapted to $(\F_n)_{n\geq 1}$, as {\em a $\F_{\infty}$-decoupled sum sequence of} $f$. As before, we omit the reference to the $\sigma$-algebra if it obvious. Note that inequality \eqref{decoup} holds for all $\F_{\infty}$-decoupled sum sequences of $f$ if it holds for some $\F_{\infty}$-decoupled sum sequences of $f$ as they are identical in law.
\begin{remark}
In situations where only the laws of $(d_n)_{n\geq 1}$ and its decoupled tangent sequence $(e_n)_{n\geq1}$ are relevant, as is the case in Definition \ref{def:dec}, it suffices to consider the probability space $([0,1]^{\N}, \mathcal{B}([0,1]^{\N}), \lambda_{\N})$ where $\lambda_{\N}$ is the Lebesgue product measure. This has been demonstrated in \cite{montgomery:represmart}, where it also has been shown that one may assume the sequences $(d_n)_{n\geq 1}$ and $(e_n)_{n\geq 1}$ to have a certain structure on that probability space, which is useful when trying to gain insight in the properties of decoupled sequences. However, the details are rather technical, so we will stick to the definition involving arbitrary probability spaces. \end{remark}
A natural question to ask is whether a \mbox{(quasi-)}Banach space $X$ that satisfies the decoupling inequality for some $p\in (0,\infty)$, automatically satisfies it for all $q\in (0,\infty)$. In \cite{CoxVeraar} we have shown that if the decoupling inequality is satisfied in a Banach space $X$ for some $p\in [1,\infty)$, then it is satisfied for all $q\in (p,\infty)$. This also follows from results presented in \cite{geiss:BMO}, see Remark \ref{r:geiss}. One of the main results of this article is that the decoupling inequality is in fact satisfied for all $q\in (0,\infty)$ if it is satisfied for some $p\in (0,\infty)$, see Theorem \ref{t:allpq} in Section \ref{s:allp}. As a result of that Theorem \ref{t:allpq} we may speak of a \mbox{(quasi-)}Banach space $X$ for which {\em the decoupling inequality holds}, meaning a space for which it holds for some, and hence all, $p\in (0,\infty)$.\par
A necessary condition for a Banach space to satisfy the decoupling inequality is that $X$ has finite cotype. This has been proven in \cite[Theorem 2]{Garling:RMTI}, see also \cite[Example 3]{CoxVeraar}, by proving that $c_0$ does not satisfy the decoupling inequality and then appealing to the Maurey-Pisier theorem. In fact, by Lemma \ref{lem:local} the decoupling property is local: if a Banach space $X$ satisfies the decoupling inequality for some $p\in (0,\infty)$ and Banach space $Y$ is finitely representable in $X$ then $Y$ satisfies the decoupling inequality for $p$, and $D_p(Y)\leq D_p(X)$. Moreover, we have demonstrated in \cite{CoxVeraar} that $D_p(L^p(S;Y)) = D_p(Y)$ whenever $Y$ is a Banach space satisfying the decoupling inequality for some $p\in [1,\infty)$. In Section \ref{s:allp} we show that this extends to $p\in (0,\infty)$, thus the $L^p$-spaces with $p\in (0,\infty)$ satisfy the decoupling inequality. This indicates that quasi-Banach spaces are a natural setting in which to study the 
decoupling inequality.\par
Although we refer to inequality \eqref{decoup} as \emph{the} decoupling inequality, various other types of decoupling inequalities have been studied.  Below we shall elaborate on some related inequalities and results, in particular we will shall consider the (randomized) \textsc{umd} inequality. Now we only wish to mention that every \textsc{umd} space satisfies the decoupling inequality (see Corollary \ref{cor:umd}), but the reverse does not hold, as $L^1$ is not a \textsc{umd} space.
\subsection*{Decoupling and vector-valued stochastic integrals}
We mentioned earlier that the decoupling inequality \eqref{decoup} is of interest to us for its applications to the theory of stochastic integration in Banach spaces. In fact, if the decoupling inequality holds in $X$ for some $p\in (0,\infty)$, then one obtains one-sided Burkholder-Davis-Gundy type inequalities for the $p^{\textrm{th}}$ moment of an $X$-valued stochastic process. We will demonstrate this in Section \ref{s:bdg}, where we also show that one obtains two-sided Burkholder-Davis-Gundy type inequalities from the two-sided decoupling inequality \eqref{decoup-2} below.\par
Using decoupling inequalities to prove Burkholder-Davis-Gundy type inequalities is not a novel idea. E.g.\ in \cite{Garling:BMandUMD} and \cite{vanNeervenVeraarWeis} Burkholder-Davis-Gundy type inequalities are obtained from (randomized) \textsc{umd} inequalities. Using their approach one does not obtain Burkholder-Davis-Gundy type inequalities for $p^{\textrm{th}}$ moments when $p\in (0,1]$. Moreover, the approach in \cite{Garling:BMandUMD} requires the stochastic process in the integrand to be adapted to the filtration generated by the Brownian motion, which we do not wish to assume.\par
In the recent work by Dirksen \cite{Dirk:12} the decoupling inequality \eqref{decoup} and its reverse estimate have been used to obtain moment estimates for $L^p$-valued Poisson stochastic integrals. As such estimates could not be obtained directly from the (randomized) \textsc{umd} inequalities, the decoupling inequality \eqref{decoup} seems to be the `right' inequality in the context of stochastic integration.
\subsection*{Best constants in the decoupling inequality}
The behavior of the decoupling constant $D_p(X)$ in \eqref{decoup} is of interest as it can be used to obtain the right (optimal) behavior of the constants in inequalities such as the Burkholder-Davis-Gundy inequality and
the Rosenthal inequality for martingale difference sequences as $p$ tends to
$\infty$ (see \cite[Theorem 7.3.2]{delaPG}). In
\cite{hitczenko:decoupineq} Hitczenko (also see \cite[Chapter 7]{delaPG}) proves the
remarkable result that if $X=\R$ then there is a universal constant $D_{\mathbb{R}}$ such
that \eqref{decoup} holds for all $p\in [1, \infty] $ with $D_p=D_{\R}$, i.e.
\begin{align}\label{uniformR}
D_{\mathbb{R}} &:= \sup_{p\in [1,\infty]} D_p(\mathbb{R}) < \infty.
\end{align}
In \cite{HitMS} the existence of a universal constant in \eqref{decoup} has been proven with the $L^p$-norms replaced by a large class
of Orlicz norms and rearrangement invariant norms on $\O$. The traditional approach to proving extrapolation results is by methods as introduced in \cite{BurkholderGundy_ExIn}. In \cite{geiss:BMO}, such extrapolation results have been stated in a BMO-framework with which one obtains estimates in a more general setting (see also Remark \ref{r:geiss}).\par
We will show that if $H$ is a Hilbert space, then $\sup_{p\in [1,\infty)}D_p(H)\leq D_{\mathbb{R}}$, where $D_{\mathbb{R}}$ is defined as in equation \eqref{uniformR} (see Corollary \ref{cor:uniformH}). It remains an open problem whether the constant $D$ in \eqref{decoup} can be taken independently of $p$ in the general case that $X$ is a \mbox{(quasi-)}Banach space. There is some hope that the use of Burkholder functions as in \cite{hitczenko:notes,mcconnell:decoup} could lead to results in this direction. However, in our situation nonsymmetric Burkholder functions are needed.
\subsection*{Other decoupling inequalities}\label{ss:de}
The decoupling inequality we consider is closely related to the Banach space property called \textsc{umd} (Unconditional convergence of Martingale Difference sequences). The class of \umd Banach spaces has been introduced by Burkholder in
\cite{Burkholder:Geom} (see also \cite{Bu3} for an overview) and has proven to be useful when extending classical harmonic
analysis \cite{Bou2,Fig90,Zim} and stochastic integration \cite{mcconnell:decoup,vanNeervenVeraarWeis} to the vector-valued
situation. In \cite{Hnonhom} the following equivalence has been used to obtain an extension of the nonhomogeneous $Tb$-theorem in \cite{NazTrVo} to the vector-valued setting: $X$ is a \textsc{umd} Banach space if and only if for all
(for some) $1<p<\infty$ there exist constants $C_p$ and $D_p$ such that one has:
\begin{align}\label{decoup-2}
C_p^{-1}\|g_n\|_p \leq \|f_n\|_p \leq D_p \|g_n\|_p,
\end{align}
for all $n\geq 1$ and all $X$-valued $L^p$-martingales $f$ adapted to some filtration $(\F_n)_{n\geq 1}$, and any $g$ that is a
decoupled sum sequence of $f$. The least constants for which \eqref{decoup-2} holds are denoted by $C_p(X)$ and $D_p(X)$. This equivalence result has been proven by both Hitczenko \cite{hitczenko:notes} and McConnell \cite{mcconnell:decoup}, and from their proofs it follows that $\max\{C_p(X),D_p(X)\}\leq \beta_p(X)$ where $\beta_p(X)$ is the \textsc{umd} constant of $X$.\par
The second inequality in \eqref{decoup-2} corresponds to the decoupling inequality \eqref{decoup} for $p\in (1,\infty)$, the only difference being that $f$ in \eqref{decoup-2} is assumed to be a $(\F_n)_{n\geq 1}$-martingale. It follows from Lemma \ref{l:condsym} below that this difference is artificial. For this we use that every $(\F_n)_{n\geq 1}$-adapted sequence $(f_n)_{n\geq 1}$ in $L^p(\Omega,X)$ with $p\in [1,\infty)$ such that $f_n-f_{n-1}$ is $\F_{n-1}$-conditionally symmetric is a martingale difference sequence. The reason we choose not to work with martingale difference sequences is that they are not well-defined for $p<1$.\par
The inequalities \eqref{decoup-2} allow for a way to `split' the \umd property into two weaker properties. The aforementioned randomized \umd spaces, which have been introduced in \cite{Garling:RMTI}, are obtained by `splitting' the \umd property in a different way, leading to the following inequalities:
\begin{align}\label{LERMT}
[\beta^{+}_p]^{-1} \Big\n \sum_{k=1}^{n}  \eps_k d_k \Big\n_p \leq \Big\n \sum_{k=1}^{n} d_k \Big\n_p \leq  \beta^{-}_p \Big\n \sum_{k=1}^{n}  \eps_k d_k \Big\n_p;\quad n\geq 1,
\end{align}
where $(d_k)_{k\geq 1}$ is an $X$-valued martingale difference sequence, $(\eps_k)_{k\geq 1}$ is a Rademacher sequence independent of $(d_k)_{k\geq 1}$, and $\beta^{-}_p,\beta^{+}_p$ are constants independent of $(d_k)_{k\geq 1}$ and $n$. It is an open question whether there exists a Banach space $X$ that fails to be a \textsc{umd} space but for which the first inequality in \eqref{LERMT} holds for all $X$-valued martingale difference sequence. However, it was demonstrated in \cite{geiss:counterexample} that for $p$ fixed there fails to exist a constant $c$ such that $\beta_p(X) \leq c C_p(X)$ for all Banach spaces $X$.\par
Note that the inequalities \eqref{LERMT} coincide with the inequalities \eqref{decoup-2} if one considers only those $f$ which are adapted to the dyadic filtration (Paley-Walsh martingales). However, in general the inequalities are different. For $X=\R$ we already mentioned that the constant in \eqref{decoup} is bounded as $p\rightarrow \infty$. However, the optimal constant for the second inequality in \eqref{LERMT} is $\mathcal{O}(\sqrt{p})$ as $p\rightarrow \infty$. Indeed, by the Khintchine inequalities there is a constant $C$ such that
\begin{align*}
\Big\n \sum_{k=1}^{n} d_k \Big\n_p &\leq  \beta^{-}_p(\R) \Big\n \sum_{k=1}^{n}  \eps_k d_k \Big\n_p
\leq C \sqrt{p} \beta^{-}_p(\R) \Big\n \Big(\sum_{k=1}^{n}  |d_k|^2\Big)^{\frac12} \Big\n_p.
\end{align*}
Since the best constant in the above square function inequality is $p-1$ if $p\geq 2$ (see \cite[Theorem 3.3]{Burkeplo}) we deduce that $C \sqrt{p} \beta^{-}_p(\R)\geq p-1$ and therefore the above claim follows.
Furthermore, there is an example in \cite{Garling:RMTI} showing that there exist Banach lattices with finite cotype that do \emph{not} satisfy \eqref{LERMT}. However, the martingales constructed to prove this are \emph{not} Paley-Walsh martingales, which means there is hope that all Banach lattices with finite cotype satisfy the decoupling inequality \eqref{decoup}. (For the definition and theory of type and cotype we refer the reader to \cite{DiesJarTon}.)\par
The monograph \cite{delaPG} and the references therein provide a good overview of the various decoupling inequalities that have been studied. For the case that $X=\R$ the decoupling inequality \eqref{decoup} has been studied among others by De la Pe\~na, Gin\'e, Hitczenko and Montgomery-Smith, see
\cite[Chapter 6 and 7]{delaPG}, \cite{hitczenko:decoupineq} and \cite{HitMS}. Another important decoupling inequality that is studied in \cite[Chapters 3-5]{delaPG} has been proven to hold in all Banach spaces \cite{delaPena94}, whereas it has been demonstrated by Kalton \cite{kalton_Raddecoup} that it fails in some \mbox{quasi-}Banach spaces.
\section{Random sequences in quasi-Banach spaces}\label{s:pre}
As explained in the introduction we consider decoupling inequalities in the
setting of quasi-Banach spaces. The definition of a quasi-Banach space is identical to that of a Banach space, except that the triangle inequality is replaced by
$$\|x+y\|\leq C(\|x\|+\|y\|),$$
for all $x,y\in X$, where $C$ is some constant independent of $x$
and $y$. We shall only need some basic results on such spaces, and we refer the reader to \cite{Kalton_handbook} and references therein for general theory and more advanced results.
\begin{defn}\label{def:rnorm}
Let $X$ be a quasi-Banach space. We say that $X$ is an $r$-normable quasi-Banach space for some $0<r\leq 1$ if there exists a constant $C>0$ such that
\begin{align*}
\Big\n \sum_{j=1}^{n} x_j \Big\n^{r} &\leq C \sum_{j=1}^{n} \n x_j \n^{r}
\end{align*}
for any sequence $(x_j)_{j=1}^n \subseteq X$.
\end{defn}
The space $L^p(0,1)$ with $p\in (0,1)$ is an example of a
$p$-normable quasi-Banach space. In fact, by the Aoki-Rolewicz Theorem  \cite{Aok42},
\cite{Rol57}, any quasi-Banach space $X$ may be equivalently renormed
so it is $r$-normable for some $r\in (0,1]$, with $C=1$. It easily follows that every quasi-Banach space is a (not necessarily locally convex) $F$-space. Whenever we speak of an $r$-normable quasi-Banach space $X$ in this article, we implicitly assume $C=1$. Observe that if $X$ is a $r$-normable quasi-Banach space and $x,y\in X$ then
\begin{align}\label{reverseT}
| \n x\n^{r} - \n y\n^{r} | &\leq \n x -y \n^{r},
\end{align}
and hence the map $x\rightarrow \n x \n$ is continuous and therefore Borel
measurable.\par

Recall that an $X$-valued random variable is a Borel measurable mapping from $\O$ into $X$ with separable range (see \cite[Section I.1.4]{VaTaCho}). We say that an $X$-valued random variable $\xi$ on the probability space $(\Omega,\mathcal{A},\P)$ with $\G\subseteq \mathcal{A}$ a $\sigma$-algebra is \emph{$\G$-conditionally symmetric} if for all $B\in \mathcal{B}(X)$ one has $\P(\xi\in B \,|\, \G)=\P(-\xi\in B \,|\, \G)$. We sometimes omit the $\sigma$-algebra if it is obvious from the context.\par
Also recall the following notation: if $(\zeta_i)_{i\in I}$ is a set of $X$-valued random
variables indexed by an ordered set $I$, then $\zeta_i^*=\sup_{j\leq i} \Vert \zeta_j \Vert$ and
$\zeta^* = \sup_{j\in I} \Vert \zeta_j\Vert$. \par
Let $(\O, {\mathcal A}, \P)$ be a complete probability space. We recall some
probabilistic lemmas to be used later on. Since we need them in the
quasi-Banach setting, we provide the short proofs which might be well-known to
experts. The following version of L\'evy's inequality holds in quasi-Banach spaces:

\begin{lemma}\label{l:levy}
Let $X$ be an $r$-normable quasi-Banach space. Let $\mathcal{G}\subseteq \mathcal{A}$ be a sub-$\sigma$-algebra. Let $(\xi_k)_{k=1}^{n}$ be a sequence of $\mathcal{G}$-conditionally independent and $\mathcal{G}$-conditionally symmetric $X$-valued random variables. Then for all $t>0$ one
has:
\begin{align*}
\P\Big(\max_{k=1,\ldots,n}\Big\n \sum_{j=1}^{k}\xi_j \Big\n > t\,\Big|\,\mathcal{G}\Big) &\leq 2
\P\Big(\Big\n\sum_{j=1}^{n}\xi_j \Big\n > 2^{1-\inv{r}}t\,\Big|\,\mathcal{G}\Big)
\end{align*}
and
\begin{align*}
\P\Big(\max_{k=1,\ldots,n}\big\n \xi_j \big\n > t\,\Big|\,\mathcal{G}\Big) &\leq 2
\P\Big(\Big\n\sum_{j=1}^{n}\xi_j \Big\n > 2^{1-\inv{r}}t\,\Big|\,\mathcal{G}\Big).
\end{align*}
\end{lemma}

Let $\xi = (\xi_1, \ldots, \xi_n)$. For the proof note that there is a regular version $\mu:\O\times\mathcal{B}(X^n)\to [0,1]$ of $\P(\xi\in \cdot\,|\,\mathcal{G})$, with the following properties: for all $\omega\in \O$, $\mu(\omega,\cdot)$ is a probability measure on
$(X^n,\mathcal{B}(X^n))$ and for all $B\in \mathcal{B}(X^n)$ one has $\mu(\cdot,B)
=\P(\xi\in B\,|\,\mathcal{G})$ a.s.\ (see \cite[Theorems 6.3 and 6.4]{kallenberg}). Moreover, for any Borel function $\phi:X^n\to \R_+$ one has: \begin{equation*}
\int_{X^n} \phi(x) \, \mu(\omega,dx) = \E(\phi(\xi)\,|\, \mathcal{G})(\omega), \ \text{for almost all $\omega\in \O$},
\end{equation*}
whenever the latter exists. For the existence of the regular version note that $X^n$ is a separable complete metric space;  $\xi_j$ is separably valued for each $1\leq j \leq n$ and the metric is given by $d(x,y)=\n x-y \n^r$. The regular version $\mu$ of the conditional probability can be used to reduce the proof of the lemma to the case without conditional
probabilities. Indeed, let $\tilde{\xi}=(\tilde{\xi}_j)_{j=1}^n:X^n\to X^n$ be given by $\tilde{\xi}(x) = x$. Then one can argue with the random variable $\tilde{\xi}$ and probability measure $\mu(\omega,\cdot)$ on $X^n$ with $\omega\in \O$ fixed. We use this method below.
\begin{proof}
We only give a proof for the first statement, which is a modification of the
proof given both in \cite[Theorem 1.1.1]{delaPG} and in \cite[Proposition
1.1.1]{KwWo}. These monographs also provide a proof of the second statement
which is very similar to that of the first.\par As explained before the lemma we can leave out the conditional probabilities. For $k=1,\ldots,n$ define $S_k=\sum_{j=1}^{k}\xi_j$ and
$$A_k=\{\n S_j \n \leq t \textrm{ for all } j=1,\ldots,k-1; \n S_k \n>t \}.$$
Note that the sets $A_k$, $k=1,\ldots,n,$ are mutually disjoint. Define
$S_n^{(k)}:=S_k-\xi_{k+1}-\ldots -\xi_n$. Observe that by symmetry and
independence of the random variables $(\xi_k)_{k=1}^{n}$ the random variables
$S_n$ and $S_n^{(k)}$ have the same conditional distribution with respect to
$\sigma\big(\bigcup_{j=1}^{k}\xi_j\big)$. Hence
\begin{align*}
\P(A_k\cap \{\n S_n \n > 2^{1-\inv{r}}t\})&=\P(A_k\cap \{\n S^{(k)}_n \n >
2^{1-\inv{r}}t\}).
\end{align*}
On the other hand, because for any $x,y\in X$ one has $\n x \n \leq
2^{\inv{r}-1}\max\{\n x+y\n,\n x-y\n\}$,
on the set $A_k$ one has $t<\n S_k\n \leq 2^{\inv{r}-1}\max\{\n S_n\n,\n
S_n^{(k)}\n\}$ and thus
\begin{align*}
A_k=(A_k\cap \{\n S_n \n > 2^{1-\inv{r}}t\}) \cup (A_k\cap \{\n S^{(k)}_n \n >
2^{1-\inv{r}}t\}).
\end{align*}
Therefore
\begin{align*}
\P(S_n^*>t)&= \sum_{k=1}^{n}\P(A_k) \leq 2\sum_{k=1}^{n}\P(A_k\cap \{\n S_n \n > 2^{1-\inv{r}}t\})
\\ &= 2\P\big(\bigcup_{k=1}^{n}A_k \cap \{\n S_n \n > 2^{1-\inv{r}}t\}\big) \leq 2\P(\n S_n \n > 2^{1-\inv{r}}t).
\end{align*}
\end{proof}

As a consequence we obtain the following peculiar result which we need twice below. It is a ``toy"-version of the Kahane contraction principle.
\begin{corol}\label{cor:kahanecontractionalmost}
Assume the conditions of Lemma \ref{l:levy} hold. Let $(v_j)_{j=1}^n$ be a $\{0,1\}$-valued sequence of random variables such that $(v_j \xi_j)_{j=1}^n$ is again $\mathcal{G}$-conditionally independent and $\mathcal{G}$-conditionally symmetric. Then for all $t\geq 0$
one has:
\[
\P\Big(\Big\n \sum_{j=1}^{n} v_j \xi_j \Big\n > t\,\Big|\,\mathcal{G}\Big) \leq 2
\P\Big(\Big\n\sum_{j=1}^{n}\xi_j \Big\n > 2^{1-\inv{r}}t\,\Big|\,\mathcal{G}\Big).
\]
\end{corol}
Nigel Kalton kindly showed us how to obtain a Kahane contraction principle for tail probabilities for $r$-normable quasi-Banach space. However, the standard convexity proof for $r=1$ (cf.\ \cite[Corollary 1.2.]{KwWo}) does not extend to the case $r<1$, and the constants are more complicated. Since we do not need the more general version we only consider the situation of Corollary \ref{cor:kahanecontractionalmost}.
\begin{proof}
As in Lemma \ref{l:levy}, using a regular conditional probability for the $X^{2n}$-valued random variable
$\big((v_j \xi_j)_{j=1}^n, (\xi_j)_{j=1}^n\big)$, one can reduce to the case without conditional probabilities.

Let $(\eps_k)_{k\geq 1}$ be a Rademacher sequence on an independent complete probability space, where $\E_{\eps}$ and $\P_{\eps}$ denote the
the expectation and probability measure with respect to the Rademacher sequence. We obtain:
\begin{align*}
\P\Big(\Big\n\sum_{j=1}^{n} v_j \xi_j \Big\n>t \Big)
& \stackrel{(i)}{=} \E_{\eps} \E \one_{\{\n \sum_{j=1}^{n} \epsilon_j v_j \xi_j \n>t\}}
= \E \P_{\epsilon}\Big(\Big\n \sum_{j=1}^{n} \eps_j v_j \xi_j \Big\n>t \Big)
\\ & \stackrel{(ii)}{\leq}  2 \E \P_{\epsilon}\Big(\Big\n \sum_{j=1}^{n} \epsilon_j \xi_j \Big\n>2^{1-\inv{r}} t \Big)
\stackrel{(iii)}{=} 2\P\Big(\Big\n\sum_{j=1}^{n} \xi_j \Big\n>2^{1-\inv{r}} t \Big),
\end{align*}
where $\one$ denotes the indicator function.
In (i) and (iii) we used the independence and symmetry of $(v_j\xi_j)_{j=1}^n$ and of $(\xi_j)_{j=1}^n$. In (ii) we applied Lemma \ref{cor:kahanecontractionalmost}
to the random variables $(\epsilon_j v_j(\omega) \xi_j(\omega))_{j=1}^n$ where $\omega\in \O$ is fixed
and we used that $v_j \in \{0,1\}$.
\end{proof}

Recall that for $a,b\geq 0$ and $p\geq 1$ one has:
\begin{align*}
a^p+b^p & \leq (a+b)^p \leq 2^{p-1}(a^p+b^p),
\end{align*}
the latter inequality following by convexity. For $0<p\leq 1$ the reversed
inequalities hold, hence by defining
\begin{align}\label{ul}
l_p:= 2^{1-p}\maxsym 1 \quad \textrm{and} \quad u_p:=2^{p-1}\maxsym 1, \ \
p\in (0,\infty),
\end{align}
we obtain the following general statement for $p\in (0,\infty)$ and $a,b\geq
0$:
\begin{align}\label{pest}
l^{-1}_p(a^p+b^p) & \leq (a+b)^p \leq u_p(a^p+b^p).
\end{align}
Note that $2^{1-p}u_p=l_p$. A tiny yet useful Lemma:
\begin{lemma}\label{l:symsum}
Let $X$ be an $r$-normable quasi-Banach space and let $\mathcal{G}\subseteq \mathcal{A}$ be a sub-$\sigma$-algebra. Let $\xi$ and $\zeta$ be $\mathcal{G}$-conditionally independent
$X$-valued random variables. If $\zeta$ is $\mathcal{G}$-conditionally symmetric, then for all $p\in
(0,\infty)$ one has:
\begin{align*}
\E[\n \xi \n^p\,|\,\mathcal{G}] \leq 2^{1-p}u_{p/r}\E[\n \xi+\zeta \n^p\,|\,\mathcal{G}],
\end{align*}
where $u_{p/r}$ is as defined in \eqref{ul}.
\end{lemma}
\begin{proof}
As in Lemma \ref{l:levy} it suffices to prove the estimate without conditional expectations.

Because $\xi$ and $\zeta$ are independent and $\zeta$ is symmetric, $\xi+\zeta$ and $\xi-\zeta$ are
identically distributed. By \eqref{pest} one has:
\begin{align*}
\E \n \xi \n^{p} &\leq 2^{-p}\E (\n \xi +\zeta\n^r +\n \xi-\zeta \n^r)^{\frac{p}{r}} \\
& \leq 2^{-p}u_{p/r}\E(\n \xi +\zeta\n^p +\n \xi-\zeta \n^p)
 = 2^{1-p}u_{p/r}\E\n \xi + \zeta\n^p.
\end{align*}
\end{proof}
From \cite[p.\ 161]{LeTa} we adapt to the quasi-Banach space setting a reverse
Kolmogorov inequality:
\begin{lemma}\label{l:revKol}
Let $X$ be an $r$-normable quasi-Banach space and let $p\in (0,\infty)$. Let
$(\xi_k)_{k=1}^{n}$ be a sequence of $\mathcal{G}$-conditionally independent and $\mathcal{G}$-conditionally symmetric $X$-valued random
variables. Then for all $t>0$ one has:
\begin{align*}
\P\Big(\max_{k=1,\ldots,n}\Big\n\sum_{j=1}^{k}\xi_j \Big\n > t\,\Big|\,\mathcal{G}\Big)&\geq
2^{p-1}\left[ u^{-2}_{p/r}-\frac{t^p+\E (|\xi^*|^p\,|\,\mathcal{G})}{\E (\n \sum_{j=1}^{n}\xi_j
\n^p\,|\, \mathcal{G})}\right].
\end{align*}
\end{lemma}
In particular, if $r=1$ this corresponds to the result as stated in \cite[p.\ 161]{LeTa}.

\begin{proof}
As in the last two lemmas it suffices to consider the situation without conditioning.
Set $S_k=\sum_{j=1}^{k}\xi_j$ ($k=1,\ldots,n$), $S_0 = 0$, and define the stopping time
\begin{align*}
\tau &:=\inf\{k:\n S_k \n > t \}.
\end{align*}
On the set $\{\tau=k\}$ one has by applying \eqref{pest} twice:
\begin{align*}
\n S_n \n^{p} & \leq u_{p/r} (u_{p/r}[\n S_{k-1}\n^{p} + \n \xi_k \n^{p}] + \n S_n - S_k \n^p)\\
& \leq u_{p/r} (u_{p/r}[t^{p} + (\xi_n^*)^{p}] + \n S_n - S_k \n^p).
\end{align*}
Hence
\begin{align*}
\E \n S_n \n^{p} &\leq t^p\P(S_n^*\leq t) + \sum_{k=1}^{n}\int_{\{\tau=k\}}\n S_n \n^p d\P \\
&\leq t^p \P(S_n^*\leq t)+ u_{p/r}\sum_{k=1}^{n}\int_{\{\tau=k\}}
(u_{p/r}[t^{p} + (\xi_n^*)^{p}] + \n S_n - S_k \n^p) d\P.
\end{align*}
Because $S_n - S_k$ is independent of $\{\tau = k\}$ and
$\sum_{k=1}^{n}\P(\tau=k)=\P(S_n^*> t)$ the above can be estimated by:
\begin{align*}
\E \n S_n \n^{p} &\leq t^p \P(S_n^*\leq t)+ u_{p/r}^2\big[t^{p}\P(S_n^*>
t) + \E(\xi_n^*)^{p} \big]\\
& \quad + u_{p/r} \sup_{1\leq k\leq n}\E\n S_n - S_k
\n^p\P(S_n^*> t).
\end{align*}
By Lemma \ref{l:symsum} we have $\E\n S_n - S_k \n^p\leq 2^{1-p}u_{p/r}\E\n
S_n \n^p$ and thus, observing that $u_{p/r}\geq 1$,
\begin{align*}
\E \n S_n \n^{p} &\leq u^2_{p/r}\big[t^{p} + \E(\xi_n^*)^{p} + 2^{1-p}\E\n S_n
\n^p\P(S_n^*> t)\big],
\end{align*}
from which the desired estimate follows.
\end{proof}

The next lemma relates the distribution of $e^*$ and $d^*$ if $(e_n)_{n\geq
1}$ is a decoupled tangent sequence of $(d_n)_{n\geq 1}$ (see \cite[Lemma
1]{Hitczenko_CompM} or \cite[Theorem 5.2.1]{KwWo}):
\begin{lemma}\label{l:tanseqP}
Let $X$ be an $r$-normable quasi-Banach space and let $(e_n)_{n\geq 1}$ be a decoupled tangent sequence of $(d_n)_{n\geq 1}$. Then for each $t>0$ one has:
\begin{align*}
\mathbb{P}(e^*> t) &\leq 2\mathbb{P}(d^*> t) &\textrm{and}&& \mathbb{P}(d^*> t) &\leq 2\mathbb{P}(e^*> t).
\end{align*}
\end{lemma}
(The proof requires no adaptation; if $(e_n)_{n\geq 1}$ is a decoupled tangent sequence of $(d_n)_{n\geq 1}$
then the sequence $(\n e_n\n)_{n\geq 1}$ is a decoupled tangent sequence of $(\n d_n\n)_{n\geq
1}$.)\par

The following lemma is well-known to experts, but we could not find a reference.
\begin{lemma}\label{lem:preddcoupl}
Let $X$ be a complete separable metric space, and let $(S,\Sigma)$ be a measurable space. Suppose $(d_n)_{n\geq 1}$ is an $(\F_n)_{n\geq 1}$-adapted $X$-valued sequence and let $(v_n)_{n\geq 1}$ be an $(\F_n)_{n\geq 0}$-predictable $S$-valued sequence. For $n\geq1$ let $h_n:X\times S\rightarrow X$ be a $\mathcal{B}(X)\otimes \Sigma$-measurable function. Then $(h_n(e_n, v_n))_{n\geq 1}$ is a decoupled tangent sequence of $(h_n(d_{n},v_n))_{n\geq 1}$ whenever $(e_n)_{n\geq 1}$ is a decoupled tangent sequence of $(d_n)_{n\geq 1}$. \par
Moreover, if the function $h_n$ satisfies $-h_n(x,s)= h_n(-x,s)$ for all $x\in X, s\in S$ for some $n\geq 1$, then $h_n(d_{n},v_n)$ is $\F_{n-1}$-conditionally symmetric and $h_n(e_{n},v_n)$ is $\F_{\infty}$-conditionally symmetric whenever $d_n$ is.
\end{lemma}

\begin{proof}[Proof of Lemma \ref{lem:preddcoupl}]
Fix $k\geq 1$. Let $\mu_1,\mu_2:\O\times \mathcal{B}(X)\to [0,1]$ be regular conditional probabilities for $\P(d_k\in \cdot\,|\, \F_{k-1})$ and $\P(e_k\in \cdot\,|\, \F_{\infty})=\P(e_k\in \cdot\,|\, \F_{k-1})$. Then by the fact that $(e_n)_{n\geq 1}$ is a decoupled tangent sequence of $d_n$ we have $\mu_1(\omega,\cdot) = \mu_2(\omega,\cdot)$ for almost all $\omega\in \O$. Let $\tilde{d}_k(x) = x$ and $\tilde{e}_k(x) = x$. Let $B\subseteq X$ be a Borel set. Then by disintegration (also see \cite[Theorems 6.3 and 6.4]{kallenberg}) for almost all $\omega\in \O$ one has:
\begin{align*}
\P(h_k(d_k, v_k)\in B\,|\,\mathcal{F}_{k-1})(\omega) &= \int_{X} \one_{h_k(\tilde{d}_k(x), v_k(\omega))\in B} \,  \mu_1(\omega, dx)
\\ & = \int_{X} \one_{h_k(\tilde{e}_k(x), v_k(\omega))\in B} \,  \mu_2(\omega, dx)
\\ & = \P(h_k(e_k, v_k)\in B\,|\,\mathcal{F}_{k-1})(\omega).
\end{align*}
The claim concerning the conditional symmetry of $h_n(d_{n}, v_n)$ and $h_n(e_{n}, v_n)$ can be proven in a similar fashion.

Therefore, it remains to prove the conditional independence. Fix $n\geq 1$. Let $\mu:\O\times \mathcal{B}(X^n)\to [0,1]$ be a regular conditional probability for $(e_k)_{k=1}^n$. Let $\tilde{e}:X^n\to X^n$ be given by $\tilde{e}(x) = x$. Then for each $\omega\in\O$,  $(\tilde{e}_k)_{k=1}^n$ are independent random variables with respect to the probability measure with respect to $\mu(\omega, \cdot)$. In this part of the argument we only require that $v_n$ is $\F_{\infty}$-measurable. By disintegration one obtains that for all Borel sets $B_1, \ldots, B_n\subseteq X$ and almost all $\omega\in \O$ one has:
\begin{align*}
&\P( h_1(e_1,v_1)\in B_1, \ldots, h_n(e_n, v_n)\in B_n\,|\,\F_{\infty})(\omega) \\
&\qquad\qquad =\int_{X^n} \prod_{k=1}^n \one_{h_k(\tilde{e}_k(x), v_n(\omega)))\in B_k} \,  \mu(\omega, dx)
\\ &\qquad\qquad = \prod_{k=1}^n \int_{X^n}  \one_{h_k(\tilde{e}_k(x), v_n(\omega)))\in B_k} \,  \mu(\omega, dx) \quad \textrm{(by independence)}
\\ &\qquad\qquad = \prod_{k=1}^n \P(h_k(e_k,v_n)\in B_k\,|\,\F_{\infty})(\omega).
\end{align*}
\end{proof}

\section{Extrapolation lemmas}\label{sec:preplemma}
Throughout this section let $X$ be a fixed $r$-normable quasi-Banach space, and let $(\Omega,(\F_n)_{n\geq 1},\mathcal{A},\P)$ be a fixed complete probability space. As usual we define $\F_{\infty}=\sigma(\F_n:n\geq 1)$.  Moreover, in this section and the next $(d_n)_{n\geq 1}$ and $(e_n)_{n\geq 1}$ always denote the respective difference sequences of the sequences $(f_n)_{n\geq 1}$ and $(g_n)_{n\geq 1}$.\par
Let $\mathcal{M}_{\infty}$ be the set of all $(\F_n)_{n\geq 1}$-adapted uniformly bounded $X$-valued sequences $f=(f_{n})_{n\geq 1}$ such that $d_n$ is $\F_{n-1}$-conditionally symmetric for all $n\geq 1$ and for which there exists an $N\in \mathbb{N}$ such that $d_n = 0$ for all $n\geq N$. We define
\begin{align*}
& \mathcal{D}_{\infty}:=\{f\in \mathcal{M}_{\infty}\,:\, \textrm{there exists a $\F_{\infty}$-decoupled sum sequence $g$ of $f$}\\
& \qquad\qquad\qquad\qquad\qquad \textrm{ on the space }(\O,\mathcal{A},\P)\},\end{align*}
(It would be more precise to refer to $\mathcal{D}_{\infty}$ as $\mathcal{D}_{\infty}(\Omega,(\F_n)_{n\geq 1},\mathcal{A},\P;X)$ but we have assumed the space $X$ and the probability space to be fixed throughout this section.)\label{Minf}\par
The operator $T_p:\mathcal{D}_{\infty}\rightarrow L^0(\Omega, \mathcal{A}, \mathbb{R}_{+})$ is defined as
follows:
\begin{align*}
T_p(f)&=\Big( \mathbb{E}\Big[ \Big\Vert \sum_{k\geq 1} e_k
\Big\Vert^p\,\Big|\, \F_{\infty} \Big]\Big)^{\inv{p}},
\end{align*}
where $(e_k)_{n\geq 1}$ is a $\F_{\infty}$-decoupled tangent sequence of $(d_n)_{n\geq 1}$ on $(\Omega,\mathcal{A},\P)$. In the next remark it is shown that $T_p$ is well-defined.
\begin{remark}\label{rem:condistribution}
Although the sequence $(e_k)_{n\geq 1}$ is not uniquely defined on $(\Omega,\mathcal{A},\P)$, its conditional distribution given $\F_{\infty}$ is unique. Indeed, if $(\tilde{e}_k)_{k\geq 1}$ is another $\F_{\infty}$-decoupled tangent sequence for $(d_k)_{k\geq 1}$ on $(\Omega,\mathcal{A},\P)$, then by definition we have:
\begin{align}\label{eq:conddistribution}
\P(\tilde{e}_1\in B_1,\ldots, \tilde{e}_n\in B_n\,|\,\F_{\infty})= \P(d_1\in B_1\,|\,\mathcal{F}_{0})\cdot \ldots \cdot \P(d_n\in B_n\,|\,\mathcal{F}_{n-1}),
\end{align}
for all $n\geq 1$ and all Borel sets $B_1, \ldots, B_n$ and the same holds with $(\tilde{e}_k)_{k=1}^n$ replaced by $(e_k)_{k=1}^n$. A monotone class argument implies that for all Borel functions $\phi:X^n\to \R_+$ one has $\E[\phi(e_1, \ldots, e_n)\,|\,\F_{\infty}] = \E[\phi(\tilde{e}_1, \ldots, \tilde{e}_n)\,|\,\F_{\infty}]$. In particular, taking $\phi(x_1, \ldots, x_n) = \big\|\sum_{k=1}^n x_k\big\|^p$ it follows that $T_p(f)$ is unique. Moreover, from \eqref{eq:conddistribution} with $\tilde{e}_k$ replaced by $e_k$, $k=1,\ldots,n$, one also sees that $\E[\phi(e_1, \ldots, e_n)\,|\,\F_{\infty}]$ is $\F_{n-1}$-measurable.
\end{remark}

The following properties of $T_p$ are well-known and easy to prove:
\begin{enumerate}
\item $T_p$ is local, i.e.\ $T_p f=0$ on the set $\bigcap_{n\geq 1} \{\E[\|d_n\|\,|\,\mathcal{F}_{n-1}] =0\}$.
\item $T_p$ is monotone when $r=1$, i.e.\ $T_p(f^n) \leq T_p(f^{n+1})$ (see Lemma \ref{l:symsum}).
\item $T_p$ is predictable, i.e.\ $T_p(f^n)$ is $\F_{n-1}$-measurable (see Remark \ref{rem:condistribution}).
\item $T_p$ is quasilinear for all $p\in (0,\infty)$ and $r\in (0,1]$, and sublinear if $p\in [1,\infty)$ and $r=1$.
\end{enumerate}

For $f\in\mathcal{D}_{\infty}$ let $T^*_p(f):=\sup_{n\geq
1}T_p(f^n)$ and $\n f \n :=\lim_{n\rightarrow \infty}\n f_n \n$, both of which
are well-defined by definition of $\mathcal{D}_{\infty}$. Observe that if $g$ is a decoupled sum sequence of $f$ then $\n
T_p(f)\n_p =  \n g \n_p$.

The first lemma we prove employs the well-known Burkholder stopping-time
technique (see for example \cite{BurkholderGundy_ExIn}, \cite{Burkholder:Geom}). The assumption given by \eqref{pallpass} below can be interpreted as a BMO-condition, this approach has been introduced in \cite{geiss:BMO}.

Let $\tau$ be an $(\F_n)_{n \geq 1}$-stopping time and $f$ an $(\F_n)_{n \geq 1}$-adapted
sequence. The stopped sequence $f^{\tau}$ is defined by $f^{\tau}:=(\one_{\{\tau \geq
n\}}d_n)_{n \geq 1}$ and the started sequence by $^{\tau}\!\!f:=(\one_{\{\tau
< n\}}d_n)_{n \geq 1}$. If $\nu$ is another stopping time then
$^{\tau}\!\!f^{\nu}:=f^{\nu} - f^{\tau}$. It follows from Lemma \ref{lem:preddcoupl} that $^{\tau}\!\!f^{\nu}\in \mathcal{D}_{\infty}$ whenever $f\in \mathcal{D}_{\infty}$. (Thus in particular $T_p(^{\tau}\!\!f^{\nu})$ is well-defined if $f\in \mathcal{D}_{\infty}$.)

\begin{lemma}\label{l:allp1}
Let $p\in (0,\infty)$ and let $\mathcal{D}_{\infty}$ be as defined above. Suppose
that for some $b\in (0,1)$ and $A>0$ we have:
\begin{align}\label{pallpass}
\sup_{f\in \mathcal{D}_{\infty}}\sup_{0\leq k\leq l}\sup_{B\in \F_k, B\neq \varnothing}\P(\n ^{k}\!\!f^{l} \n > A\n T_p(^{k}\!\!f^{l})\n_{\infty}\,|\, B)<b.
\end{align}
Let $\beta,\delta>0$ satisfy
\begin{equation}\label{betadelta_ass}
\beta^{\varrho} - 1 = (2A^{\varrho}+1)\delta^{\varrho},
\end{equation}
where $\varrho:=\min\{r,p\}$. Then for all $f\in \mathcal{D}_{\infty}$ one has
\begin{align}\label{tallph1}
\P(f^*\geq \beta \lambda, T^*_p(f)\maxsym d^* < \delta\lambda) &\leq b \P(f^*> \lambda), \quad \lambda>0.
\end{align}
\end{lemma}
The proof is quite standard. For convenience of the reader we give the details.
\begin{proof}
Let $\beta,\delta>0$ satisfy \eqref{betadelta_ass}. Let $f\in \mathcal{D}_{\infty}$ and let $\lambda>0$ be
arbitrary. Define the following stopping times:
\begin{align*}
\mu &= \inf\{n\geq 1: \Vert f_n \Vert > \lambda\};\\
\nu&=  \inf\{n\geq 1: \Vert f_n \Vert > \beta\lambda\};\\
\sigma &= \inf\{n\geq 1: T_p(f^{n+1}) \maxsym \Vert d_n \Vert > \delta\lambda\}.
\end{align*}
On the set $\{\nu < \infty, \sigma=\infty \}$ one has by \eqref{reverseT} that:
\begin{align*}
\Vert ^\mu\!\!f^{\nu \minsym \sigma} \Vert^{\varrho} &\geq \Vert f^{\nu \minsym \sigma}
\Vert^{\varrho} - \Vert f^{\mu -1} \Vert^{\varrho} - \Vert d_{\mu}\Vert^{\varrho} \\
&  > (\beta\lambda)^{\varrho}-\lambda^{\varrho}-(\delta\lambda)^{{\varrho}} = (\beta^{\varrho}-1-\delta^{\varrho})\lambda^{\varrho}.
\end{align*}
We show that
\begin{align}\label{eq:claimtussendoor}
\Vert T_p(^\mu\!\!f^{\nu \minsym \sigma})\Vert_{\infty} & \leq
2^{\inv{\varrho}}\delta\lambda.
\end{align}
On the set $\{ \mu \geq \sigma\}$ one has $T_p(^\mu\!\!f^{\nu \minsym \sigma})=0$. On the set $\{ \mu < \sigma\}$ one has:
\begin{align*}
[T_p(^\mu\!\!f^{\nu \minsym \sigma})]^{\varrho}&=[T_p(^{\mu \minsym \sigma}\!\!f^{\nu \minsym
\sigma})]^{\varrho} = [T_p(f^{\nu \minsym \sigma} - f^{\mu \minsym \sigma})]^{\varrho} \\
& \leq [T_p(f^{\nu \minsym \sigma})]^{\varrho} + [T_p(f^{\mu \minsym \sigma})]^{\varrho},
\end{align*}
using that if $X$ is $r$-normable, then $L^p(\Omega,X)$ is $\varrho$-normable.
By definition of $\sigma$ one has $T_p(f^{\nu \minsym
\sigma}) \leq \delta \lambda$ and
$T_p(f^{\mu \minsym \sigma}) \leq \delta \lambda$, from which \eqref{eq:claimtussendoor} follows.

We obtain:
\begin{align}\label{tallph2}
&\P(f^*> \beta \lambda, T^*_p(f)\maxsym d^* \leq \delta\lambda) = \P(\nu< \infty, \sigma=\infty) \notag\\
&\qquad \qquad \leq \P(\Vert ^\mu\!\!f^{\nu \minsym \sigma} \Vert > (\beta^{\varrho}-1-\delta^{\varrho})^{\inv{\varrho}}\lambda) \notag\\
&\qquad \qquad \leq \P(\Vert ^\mu\!\!f^{\nu \minsym \sigma} \Vert >
2^{-\inv{\varrho}}\delta^{-1}(\beta^{\varrho}-1-\delta^{\varrho})^{\inv{\varrho}}\Vert T_p(^\mu\!\!f^{\nu \minsym
\sigma})\Vert_{\infty}).
\end{align}
As $\beta$ and $\delta$ satisfy \eqref{betadelta_ass} we have $A= 2^{-\inv{\varrho}}\delta^{-1}(\beta^{\varrho}-1-\delta^{\varrho})^{\inv{\varrho}}$. Thus it follows from assumption \eqref{pallpass} that
\begin{align*}
&\P(\Vert ^\mu\!\!{f}^{\nu \minsym \sigma} \Vert > 2^{-\inv{\varrho}}\delta^{-1}(\beta^{\varrho}-1-\delta^{\varrho})^{\inv{\varrho}}\Vert T_p(^\mu\!\!{f}^{\nu \minsym
\sigma})\Vert_{\infty}\,|\,\mu<\infty) \\
& \qquad \qquad = \P( ^\mu\!\!f^{\nu \minsym \sigma} > A \n T_p(^\mu\!\!f^{\nu \minsym \sigma})\n_{\infty} \,|\, \mu<\infty) \\
& \qquad \qquad = \P(\mu<\infty)^{-1} \sum_{k=1}^{\infty} \P(^\mu\!\!f^{\nu \minsym \sigma} > A \n T_p(^\mu\!\!f^{\nu \minsym \sigma})\n_{\infty} \,|\, \mu=k)\P(\mu=k)\\
& \qquad \qquad \leq b \P(\mu<\infty)^{-1} \sum_{k=1}^{\infty} \P(\mu=k) = b.
\end{align*}
As $^\mu\!\!f^{\nu \minsym \sigma}=0$ on $\{\mu = \infty\}$ we have:
\begin{align*}
\Vert ^\mu\!\!f^{\nu \minsym \sigma} \Vert \leq
2^{-\inv{\varrho}}\delta^{-1}(\beta^{\varrho}-1-\delta^{\varrho})^{\inv{\varrho}}\Vert T_p(^\mu\!\!f^{\nu \minsym
\sigma})\Vert_{\infty}
\end{align*}
on that set. Combining the above we obtain:
\begin{align*}
& \P(\Vert ^\mu\!\!f^{\nu \minsym \sigma} \Vert >
2^{-\inv{\varrho}}\delta^{-1}(\beta^{\varrho}-1-\delta^{\varrho})^{\inv{\varrho}}\Vert T_p(^\mu\!\!f^{\nu \minsym
\sigma})\Vert_{\infty}) \\
& \qquad \qquad \qquad \qquad \qquad \qquad \leq b\P(\mu<\infty) = bP(f^*> \lambda),
\end{align*}
which, when inserted in equation \eqref{tallph2}, gives \eqref{tallph1}.
\end{proof}

\begin{remark}\label{r:geiss}
Suppose $X$ is a Banach space, i.e.\ $r=1$. In \cite{geiss:BMO} it has been demonstrated how extrapolation results can be obtained from BMO-type assumptions like \eqref{pallpass} in Lemma \ref{l:allp1}.
In particular, from Corollary 6.3 and Proposition 7.3 in \cite{geiss:BMO} one can deduce that
if assumption \eqref{pallpass} is satisfied, then there exists a constant $c_{X,b,p}$ such that for all $1\leq q < \infty$ and all $f\in \mathcal{D}_{\infty}$ one has:
\begin{align*}
\n f^* \n_q \leq c_{X,b,p} q \n T_p(f) \n_q.
\end{align*}
Observe that for $q\geq p$ we have $\n T_p(f) \n_q\leq \n g\n_q$ by the conditional H\"older's inequality, where $g$ is a decoupled sum sequence of $f$. However, it seems that this approach fails when $q<p$ as well as in the more general setting that we consider in Theorem \ref{t:allpq}. Thus we proceed in a different manner.
\end{remark}

Let $\Phi:\mathbb{R}_+ \rightarrow \mathbb{R}_+$ be an
arbitrary yet fixed non-decreasing continuous function satisfying $\Phi(0)=0$.
Moreover, we assume there exists a $q\in(0,\infty)$ such that
\begin{align}\label{qPhi}
\Phi(s t) \leq s^q \Phi(t),\quad \textrm{for all } s,t\in \R_+.
\end{align}
The set of all such functions is referred to as $F_q$ in \cite{HitMS}.
\begin{prop}\label{p:allp2}
Let $p\in (0,\infty)$. Let $\mathcal{D}_{\infty}$, $\Phi$ and $q$ be as defined above, and again set $\varrho:=\min\{r,p\}$. Suppose that \eqref{pallpass} holds for some $b \in (0, 2^{-\frac{2p}{\varrho}+p-1})$ and some $A>0$.
Then for all $f\in \mathcal{D}_{\infty}$ we have
\begin{align*}
\E\Phi(f^{*}) & \leq C_{X,r,p,q} \E \Phi(\|g\|),
\end{align*}
where $g$ is a decoupled sum sequence of $f$ and $C_{X,r,p,q}$ as in \eqref{quasi_decoup_const} below. In particular, for $r=1$ and $p\geq 1$ one can take
\begin{align}\label{decoup_const}
C_{X,1,p,q}& =2^{2q +2}\Big[ 2^{p+q+1}(2^{q}+1) \left[2A  +1\right]^{q}\left[1-[2^{p+1}b]^{-\frac{1}{q}}\right]^{-q} + 1\Big].
\end{align}
\end{prop}

For a positive random variable we can write $\E \Phi(\xi) =
\int_0^\infty \P(\xi>\lambda) \, d\Phi(\lambda)$, where the integral is of
Lebesgue-Stieltjes type.

\begin{proof}
Let $f\in
\mathcal{D}_{\infty}$ be given. The Davis decomposition of $(d_n)_{n\geq1}$ is given by $d_n=d_n'+d_n''$ where $d_1':= 0$, $d_1'':=d_1$ and for $n\geq 2$:
\begin{align*}
d_n'=d_n \one_{\{\Vert d_n \Vert \leq 2d_{n-1}^*\}} \ \ \text{and} \ \
d_n''=d_n \one_{\{\Vert d_n \Vert > 2d_{n-1}^*\}},
\end{align*}
and $f_n' = \sum_{k=1}^n d_k'$ and $f_n'' = \sum_{k=1}^n d_k''$.
It follows from Lemma \ref{lem:preddcoupl} that $f',f''\in \mathcal{D}_{\infty}$ and that $\F_{\infty}$-decoupled tangent sequences of $(d_n')_{n\geq1}$ and  $(d_n'')_{n\geq1}$ are given by
$e_1':=0, e_1'':=e_1$ and for $n\geq 2$:
\begin{equation}\label{eq:davisdec}
e_n'=e_n \one_{\{\Vert e_n\Vert \leq 2 d_{n-1}^*\}}, \ \ \text{and} \ \
e_n''=e_n \one_{\{\Vert e_n\Vert > 2 d_{n-1}^*\}}.
\end{equation}
Moreover, the random variable $d_n'$ is bounded by the $\F_{n-1}$-measurable random variable
$2d_{n-1}^*$. On the other hand, for the sequence $(d_n'')_{n\geq 1}$ we have on the set $\{\n d_n \n > 2 d_{n-1}^*\}$, $n\geq 2$,
\begin{align*}
(2^{\varrho}-1)\n d_n''\n^{\varrho} + (2d_{n-1}^*)^{\varrho} \leq (2^{\varrho}-1+1)\n d_n''\n^{\varrho} \leq
2^{\varrho}(d_n^*)^{\varrho},
\end{align*}
whence $\n d_n''\n^{\varrho}\leq (1-2^{-\varrho})^{-1}[(d_n^*)^{\varrho} - (d_{n-1}^*)^{\varrho}]$ and thus
\begin{align}\label{estDavis}
\Vert f'' \Vert^r & \leq \sum_{n=1}^{\infty} \Vert d_n'' \Vert^r \leq
(1-2^{-\varrho})^{-1} (d^{*})^{\varrho}
\end{align}
(for $\varrho=1$ see \cite{Davis_intMSF} or \cite[inequality (4.5)]{Burkholder:fu}).

By \eqref{estDavis} and due to the fact that $\Phi$ is non-decreasing we have:
\begin{align*}
& \E\Phi(f^*) \leq \E\Phi(2^{\inv{\varrho}-1}[f'^* + f''^*]) \leq \E\Phi(2^{\inv{\varrho}-1}[f'^* + (1-2^{-\varrho})^{-\inv{\varrho}} d^*]) \notag\\
&\qquad \qquad  = \int_0^\infty \P(2^{\inv{\varrho}-1}[f'^* + (1-2^{-\varrho})^{-\inv{\varrho}} d^*]>\lambda) \, d \Phi(\lambda) \notag\\
&\qquad \qquad  \leq \E \Phi(2^{\inv{\varrho}}f'^*) +  \int_0^\infty \P(2^{\inv{\varrho}} d^*>(1-2^{-\varrho})^{\inv{\varrho}}\lambda) \, d \Phi(\lambda). \notag
\end{align*}
Using Lemma \ref{l:tanseqP} and the L\'evy inequality
applied conditionally (Lemma \ref{l:levy}) we can estimate the
right-most term in the above:
\begin{equation}
\begin{aligned}\label{dstarg}
&\int_0^\infty \P(d^*>2^{-\inv{\varrho}}(1-2^{-\varrho})^{\inv{\varrho}}\lambda) \, d \Phi(\lambda) \leq 2 \int_0^\infty \P(e^*>2^{-\inv{\varrho}}(1-2^{-\varrho})^{\inv{\varrho}}\lambda) \, d \Phi(\lambda)
\\ &\qquad \qquad   \leq 4 \int_0^\infty \P(\Vert g \Vert >2^{1-\frac{2}{\varrho}}(1-2^{-\varrho})^{\inv{\varrho}}\lambda) \, d \Phi(\lambda) \\
& \qquad \qquad = 4 \E \Phi(2^{\frac{2}{\varrho}-1}(1-2^{-\varrho})^{-\inv{\varrho}}\Vert g\Vert ).
\end{aligned}
\end{equation}
Therefore, we conclude that
\begin{equation}\label{eq:toproofPhi}\begin{aligned} \E\Phi(f^*) & \leq 2^{\frac{q}{\varrho}} \E \Phi(f'^*) +  2^{\frac{2q}{\varrho}-q+2}(1-2^{-\varrho})^{-\frac{q}{\varrho}} \E \Phi(\n g\n),
\end{aligned}
\end{equation}
with $q$ as in \eqref{qPhi}. It remains to estimate $\E\Phi(f'^*)$, for which we use Lemma \ref{l:allp1}.\par

Set $$\beta:= \left[2^{\frac{2p}{\varrho}-p+1}b\right]^{-\inv{q}},$$ we have $\beta>1$ because $b<2^{-\frac{2p}{\varrho}+p-1}$. Setting $ \delta^{\varrho} := (2A^{\varrho}+1)^{-1}(\beta^{\varrho}-1)$
it follows that $\beta$ and $\delta$ satisfy the conditions of Lemma \ref{l:allp1}.

We follow the proof of \cite[Lemma 2.2]{HitMS} to show that one has
\begin{equation}\label{eq:HitMS}\begin{aligned}
& \P(f'^*\geq \beta \lambda, g'^*<\delta_2 \lambda) \\ & \qquad \qquad \leq b \P(f'^*\geq
\lambda ) + \P(2d^*\geq \delta_2 \lambda)+  (1-2^{p-\frac{2p}{\varrho}}) \P(f'^*
\geq \beta\lambda),\end{aligned}
\end{equation}
where $\delta_2 := 4^{-\inv{\rho}}\delta$ and $g' = \sum_{n\geq 1} e_n'$. Indeed,
\begin{align}\label{probsplit}
\P(f'^*\geq \beta \lambda, g'^*<\delta_2 \lambda) &\leq \P(f'^*\geq \beta \lambda, T_p^*(f')<\delta \lambda, 2d^*<\delta_2\lambda) + \P(2d^*\geq \delta_2\lambda) \notag\\
&\quad +\P(f'^*\geq \beta \lambda, T_p^*(f')\geq\delta \lambda, 2d^*<\delta_2\lambda, g'^*<\delta_2 \lambda).
\end{align}
As $\delta_2\leq\delta$ it follows from the definition of $(d_n')_{n\geq 1}$
and from Lemma \ref{l:allp1} that for the first probability on the right-hand
side of \eqref{probsplit} one has:
\begin{align}\label{probsplit_h1}
\P(f'^*\geq \beta \lambda, T_p^*(f')<\delta \lambda, 2d^*<\delta_2\lambda)&\leq \P(f'^*\geq \beta \lambda, T_p^*(f')<\delta \lambda, d'^*<\delta\lambda) \notag\\
& \leq b\P(f'^*\geq \lambda).
\end{align}
In order to deal with the last probability in \eqref{probsplit}, observe that $f'^*,
d^*$ and $T^*_p(f')$ are all $\F_{\infty}$-measurable. Thus by conditioning on
$\F_{\infty}$ we obtain
\begin{align}\label{probsplit_h2}
\E[\one_{\{f'^*\geq \beta \lambda, T^*_p(f')\geq \delta\lambda, 2d^*<\delta_2\lambda\}}\P(g'^*<\delta_2\lambda\,|\,\F_{\infty})].
\end{align}
By Lemma \ref{l:revKol} we have:
\begin{align*}
\P(g'^*<\delta_2\lambda\,|\,\F_{\infty}) & \leq 1-2^{p-1}\left[ 2^{-\frac{2p}{\varrho}+2} -\frac{(\delta_2\lambda)^p+\E [(e'^*)^p\,|\,\F_{\infty}]}{\E (\n g' \n^{p}\,|\,\F_{\infty})}\right],
\end{align*}
observing that $u_{p/\varrho}=2^{\frac{p}{\varrho}-1}$ as $p\geq \varrho$.
Note that $\E (\n g' \n^{p}\,|\,\F_{\infty})= T_p(f')$ and by \eqref{eq:davisdec} we have $e'^*\leq 2d^*$, and thus on the set
\[S:=\{f'^*\geq \beta \lambda,T^*_p(f')\geq \delta\lambda, 2d^*<\delta_2\lambda\}\]
one has:
\begin{align*}
\P(g'^*<\delta_2\lambda\,|\,\F_{\infty}) & \leq 1-2^{p-1}\left[ 2^{-2\frac{p}{\varrho}+2}-\frac{2(\delta_2\lambda)^p}{(\delta\lambda)^p}\right]
= 1-2^{p-\frac{2p}{\varrho}}.
\end{align*}
Therefore we find:
\begin{align}\label{probsplit_h3}
\E[\one_{S}\P(g'^*<\delta_2\lambda\,|\,\F_{\infty})] \leq
\E[\one_{S}(1-2^{p-\frac{2p}{\varrho}})] \leq (1-2^{p-\frac{2p}{\varrho}})\P(f^*\geq
\beta\lambda).
\end{align}
Combining equations \eqref{probsplit}, \eqref{probsplit_h1},
\eqref{probsplit_h2} and \eqref{probsplit_h3} gives \eqref{eq:HitMS}.\par It
follows from \eqref{eq:HitMS} that
\begin{align*}
\P&(f'^*\geq \beta \lambda) \\ & \leq b \P(f'^*\geq \lambda ) + \P(2d^*\geq
\delta_2 \lambda)+  (1-2^{p-\frac{2p}{\varrho}}) \P(f'^* \geq \beta\lambda) +
\P(g'^*\geq \delta_2 \lambda).
\end{align*}
Collecting terms and integrating with respect to $d\Phi(\lambda)$ gives that
\begin{align*}
\E \Phi(f'^*/\beta) \leq 2^{\frac{2p}{\varrho}-p}[b \E \Phi(f'^*) + \E \Phi(2 d^*/\delta_2) +  \E\Phi(g'^*/\delta_2)].
\end{align*}
From this we see that, because $\Phi$ is non-decreasing,
\begin{align*}
\E \Phi(f'^*) & = \E \Phi(\beta f'^*/\beta )\leq \beta^{q} \E \Phi(f'^*/\beta) \\ & \leq 2^{\frac{2p}{\varrho}-p}\beta^{q}[ b \E \Phi(f'^*) + \E \Phi(2 d^*/\delta_2) + \E\Phi(g'^*/\delta_2)].
\end{align*}
Our choice of $\beta$ was such that $2^{\frac{2p}{\varrho}-p}\beta^{q} b = \inv{2} $. As $\delta_2 =4^{-\inv{\varrho}}\delta$ we have:
\begin{equation}\label{eq:HitMS4}
\begin{aligned} \E \Phi(f'^*) & \leq
2^{\frac{2p}{\varrho}-p+1}\beta^{q} [\E \Phi(2d^*/\delta_2) + \E\Phi(g'^*/\delta_2)]
 \\ &  \leq 2^{\frac{2p}{\varrho}-p+\frac{2q}{\varrho}+1} (\beta/\delta)^q [ 2^{q}\E \Phi(d^*) +   \E\Phi(g'^*)].
\end{aligned}
\end{equation}
As before in \eqref{dstarg} one can prove that $\E \Phi(d^*)\leq
2^{\frac{q}{\varrho}-q+2} \E \Phi(\Vert g\Vert$). By the L\'evy
inequality we obtain $\E\Phi(g'^*)\leq 2^{\frac{q}{\varrho}-q+1} \E\Phi(\Vert g'\Vert )$. By Corollary \ref{cor:kahanecontractionalmost} and the definition of $(e_n')_{n\geq 1}$ we have:
\begin{equation} \label{phiKahane} \begin{aligned}
\E \Phi( \n g'\n) &= \E \int_{0}^{\infty} \P\Big(\Big\n \sum_{k=1}^{n} \one_{\{\n e_k\n \leq d_{k-1}^*\}} e_k \Big\n>\lambda \,\Big| \, \F_{\infty} \Big) d\Phi(\lambda)
\\ & \leq 2 \int_{0}^{\infty} \P\Big(\Big\n \sum_{k=1}^{n} e_k \Big\n>2^{1-\inv{\varrho}}\lambda \,\Big| \,\F_{\infty} \Big) d\Phi(\lambda) \leq 2^{\frac{q}{\varrho}-q+1}\E \Phi( \n g \n).
\end{aligned} \end{equation}

Combining equations \eqref{eq:toproofPhi} and \eqref{eq:HitMS4} with the estimates above gives:
\begin{align*}
\E \Phi(f^*) &\leq C_{X,r,p,q} \E \Phi(\n g \n),
\end{align*}
for all $f\in \mathcal{D}_{\infty}$, where
\begin{align}\label{quasi_decoup_const}
C_{X,r,p,q}& = 2^{\frac{2q}{\varrho}-q +2}\big[ 2^{\frac{2p}{\varrho}-p+\frac{2q}{\varrho}-q+1}(2^{2q}+2^{\frac{q}{\varrho}})(\beta/\delta)^{\frac{q}{\varrho}} + (1-2^{-\varrho})^{-\frac{q}{\varrho}}\big],
\end{align}
and
$$ \beta/\delta = \left[2A^{\varrho} +1\right]^{\inv{\varrho}}\left[1-[2^{\frac{2p}{\varrho}-p+1}b]^{-\frac{\varrho}{q}}\right]^{-\inv{\varrho}}.$$
\end{proof}

Finally, we recall the following lemma, which can be proven like \cite[Corollary 6.4.3]{delaPG}. The inequalities in this lemma are to be interpreted in the sense that the left-hand side is finite whenever the right-hand side is so.
\begin{lemma}\label{l:condsym}
Let $X$ be an $r$-normable quasi-Banach space. Suppose that there exists a $C\geq 0$ such that for every complete probability space $(\O,(\F_n)_{n\geq 1},\mathcal{A},\P)$ and every $(\F_n)_{n\geq 1}$-adapted $X$-valued sequence $(f_n)_{n\geq 1}$, where $f_n-f_{n-1}$ is $\F_{n-1}$-conditionally symmetric for all $n\geq 1$ ($f_0\equiv 0$), and every decoupled sum sequence $g$ of $f$ we have:
\begin{align*}
\E \Phi(f_n^*) \leq C\E \Phi(g_n^*), \ \ n\geq 1.
\end{align*}
Then for every complete probability space $(\O,(\F_n)_{n\geq 1},\mathcal{A},\P)$ and every $X$-valued sequence $(f_n)_{n\geq 1}$ adapted to $(\F_n)_{n\geq 1}$ we have:
\begin{align*}
\E \Phi(f_n^*) \leq  2^{\frac{q}{r}}(2^{1+\frac{q}{r}}C +1) \E \Phi(g_n^*), \ \ n\geq 1,
\end{align*}
where $q$ is as in $\eqref{qPhi}$. The same result holds with $f_n^*$ and $g_n^*$ replaced by $f_n$ and $g_n$ in both the assumption and the assertions.
\end{lemma}

\section{$p$-Independence and the decoupling constant}\label{s:allp}
The $p$-independence of the decoupling inequality follows from taking $\Phi(s)=s^q$ in Theorem \ref{t:allpq} below.
\begin{thm}\label{t:allpq}
Let $X$ be an $r$-normable quasi-Banach space in which the decoupling inequality \eqref{decoup}
holds for some $p\in (0,\infty)$, then for $\Phi:\R_+\rightarrow \R_+$ continuous, satisfying $\Phi(0)=0$ and with $q$ as in \eqref{qPhi}, there exists a
constant $K=K_{X,r,p,q}$ such that for all complete probability spaces $(\O,\mathcal{A},(\F_n)_{n\geq 1},\P)$ and $(\F_n)_{n\geq 1}$-adapted sequences
$(f_n)_{n \geq 1}$ one has:
\begin{align}\label{decoup1q}
\E \Phi( \Vert f_n \Vert ) \leq  K \E \Phi( \Vert g_n \Vert )   \ \ \text{and} \ \
\E \Phi (f_n^*) \leq K \E \Phi(g_n^*) ,  \ \ n\geq 1,
\end{align}
where $g$ is a $\F_{\infty}$-decoupled sum sequence of $f$. \par
Now assume $X$ is a Banach space and $p\geq 1$. Then the constant $K$ can be estimated by:
\begin{align}\label{decoupconst}
K&\leq e^q 2^{\frac{3q}{p}+p+7q+7}D_p^q(X)\big(\tfrac{q}{p}\big)^{q},
\end{align}
and in particular,
\begin{align}\label{eq:geissest}
D_q(X)&\leq e 2^{\frac{3}{p}+\frac{p}{q}+7+\frac{7}{q}}D_p(X)\tfrac{q}{p},
\end{align}
for all $q\in(0,\infty)$.
\end{thm}
We interpret \eqref{decoup1q} in the sense that the left-hand side is finite whenever the right-hand side is so.
\begin{proof}
By assumption the decoupling inequality holds in the $r$-normable quasi-Banach space $X$ for some $p\in (0,\infty)$. Lemma \ref{l:condsym} states the following: If there exists a constant $C_{X,r,p,q}$ such that for every complete probability space $(\O,(\mathcal{F}_n)_{n\geq 1},\mathcal{A},\P)$ and every $(\mathcal{F}_n)_{n\geq 1}$-adapted $f$ for which $d_n$ is $\F_{n-1}$-conditionally symmetric for all $n\geq 1$ one has:
\begin{align}\label{Tallp}
\E \Phi(f^*) \leq C_{X,r,p,q} \E \Phi(\Vert g \Vert ),
\end{align}
where $g$ is a decoupled sum sequence of $f$, then \eqref{decoup1q} holds with
\begin{align}\label{Kdecoupconst}
K_{X,p,q}=K_{X,r,p,q}&\leq 2^{\frac{q}{r}}(2^{1+\frac{q}{r}}C_{X,r,p,q} +1).
\end{align}
Fix a complete probability space $(\Omega,(\F_n)_{n\geq 1},\mathcal{A},\P)$. We wish to apply Proposition \ref{p:allp2}; i.e.\ we wish to prove that assumption \ref{pallpass} is satisfied for $b\in (0,2^{-\frac{2p}{\varrho}+p-1})$ and some $A>0$ (independent of the probability space), where $\varrho=\min\{r,p\}$. Let $(f_n)_{n\geq 1}\in \mathcal{D}_{\infty}$ where $\mathcal{D}_{\infty}$ is as defined on page \pageref{Minf}, and let $g$ be a decoupled sum sequence of $f$ on $(\Omega,\mathcal{A},\P)$. Pick $0\leq k\leq l$ and let $B\in \F_k$. Observe that $T_p( ^{k}\!\!f^{l}\one_{B})=T_p( ^{k}\!\!f^{l}) \one_{B}$. By applying Chebyshev's inequality in the final line we obtain:
\begin{equation}\label{eq:bmocheck}
\begin{aligned}
\mathbb{P}(\{\Vert ^{k}\!\!f^{l} \Vert > A \Vert T_p(^{k}\!\!f^{l})\Vert_{\infty}\} \cap B) & =\mathbb{P}(\Vert ^{k}\!\!f^{l} \one_{B} \Vert > A \Vert T_p(^{k}\!\!f^{l})\Vert_{\infty}\one_{B}) \\
& \leq \mathbb{P}(\Vert ^{k}\!\!f^{l} \one_{B} \Vert > A \Vert T_p(^{k}\!\!f^{l})\one_{B} \Vert_{\infty}) \\
& \leq A^{-p}\Vert T_p(^{k}\!\!f^{l})\one_{B}\Vert_{\infty}^{-p} \Vert ^{k}\!\!f^{l} \one_{B} \Vert^{p}_{p}.
\end{aligned}
\end{equation}
By Lemma \ref{lem:preddcoupl} we have that $(^{k}\!\!g^{l}_n \one_{B})_{n\geq 1}$ is a decoupled sum sequence of $(^{k}\!\!f_n^{l}\one_{B})_{n\geq 1}$. Thus, because the decoupling inequality holds in $X$ for $p$, we have:
\begin{equation}\label{eq:decoupp}
\begin{aligned}
\Vert ^{k}\!\!f^{l} \one_{B} \Vert_{p}^{p} &\leq D_{p,X}^{p}\n ^{k}\!\!g^{l}_n \one_{B} \n_{p}^{p} = D_{p,X}^{p}\n T_p(^{k}\!\!f^{l}_n \one_{B}) \n^{p}_{p}\\
& = D_{p,X}^{p}\n T_p(^{k}\!\!f^{l}_n \one_{B}) \one_{B} \n_{p}^{p} \leq D_{p,X}^{p}\n T_p(^{k}\!\!f^{l}_n \one_{B}) \n_{\infty}^{p}\P(B).
\end{aligned}
\end{equation}
By picking $b\in (0,2^{-\frac{2p}{\varrho}+p-1})$ and setting $A = b^{-\frac{1}{p}} D_p(X)$ one obtains:
\begin{align*}
\mathbb{P}(\{\Vert ^{k}\!\!f^{l} \Vert > A \Vert T_p(^{k}\!\!f^{l})\Vert_{\infty}\} \cap B) &\leq b\P(B).
\end{align*}
Thus condition \eqref{pallpass} in Proposition \ref{p:allp2} is
satisfied, and therefore \eqref{Tallp} holds for all $f\in \mathcal{D}_{\infty}$ with a constant $C_{X,p,q,r}$ as given in that proposition. For general $(\F_n)_{n\geq 1}$-adapted sequences $(f_n)_{n\geq 1}$ with decoupled sum sequence $g$ defined on $(\O,\mathcal{A},\P)$ such that $d_n$ is $\F_{n-1}$-conditionally symmetric for all $n\geq 1$, we can reduce to the former case as follows:
\begin{equation}\label{eq:fatou}\begin{aligned}
\E \Phi(f^*)  &\stackrel{(i)}{\leq}  \liminf_{j\to \infty} \E \Phi \Big(\sup_{n\geq 1} \Big\Vert \sum_{k=1}^{n\maxsym j} d_k \one_{\|d_k\|\leq j} \Big\Vert \Big) \notag \\ & \stackrel{(ii)}{\leq} C_{X,p,q,r} \liminf_{j\to \infty} \E \Phi \Big(\Big\Vert \sum_{k=1}^{j} e_k \one_{\|e_k\|\leq j}
\Big\Vert\Big) \stackrel{(iii)}{\leq} 2^{\frac{q}{r}-q+1} C_{X,p,q,r}  \E \Phi \Vert g \Vert.
\end{aligned}\end{equation}
In (i) we used Fatou's lemma.  We applied \eqref{Tallp} in (ii), where we use that by Lemma \ref{lem:preddcoupl} $(e_k\one_{\|e_k\|\leq j})_{k=1}^n$ is a $\F_{\infty}$-conditionally symmetric decoupled tangent sequence of $(d_k\one_{\|d_k\|\leq j})_{k=1}^n$. In (iii) we used Corollary \ref{cor:kahanecontractionalmost} as in \eqref{phiKahane}.\par

We have thus proven that \eqref{Tallp} holds for an arbitrary yet fixed complete probability space, with a constant $C_{X,r,p,q}$ independent of the probability space. This completes the proof of inequality \eqref{decoup1q}.\par
In order to obtain the estimate on the constant in the decoupling inequality as provided by equation \eqref{decoupconst} it is necessary to optimize our choice of $b$. If $r=1$ and $p\geq 1$ then one can pick $b= 2^{-p-1}(1+\frac{p}{q})^{-q}$, whence $$A=2^{1+\inv{p}}(1+\tfrac{p}{q})^{\frac{q}{p}} D_p(X)\leq e2^{1+\inv{p}} D_p(X).$$ Entering this in equation \eqref{decoup_const} in Proposition \ref{p:allp2} leads to the following estimate for the constant in \eqref{Tallp}:
\begin{align}\label{eq:Cpq}
C_{X,1,p,q}\leq 2^{2q +2}\Big[ 2^{p+q+1}(2^{q}+1) \left[e2^{2+\inv{p}} D_p(X)  +1\right]^{q}\big(\tfrac{q}{p}\big)^{q} + 1\Big],
\end{align}
which, in combination with \eqref{Kdecoupconst} and some rough estimates, leads to equation \eqref{decoupconst}.\par
It may be shown that with the proof as provided here the order of the constant in equation \eqref{decoupconst} (in terms of $q$) may not be improved by choosing a different value for $b$. Moreover, one may show that a good choice for $b$ if $\varrho=\min\{r,p\}<1$ is given by $b= 2^{-\frac{2p}{r}+p-1}(1+\frac{p}{q})^{-\frac{q}{r}}$.
\end{proof}

From the proof above we obtain a somewhat stronger result, i.e.\ a maximal
inequality for conditionally symmetric adapted sequences:
\begin{corol}\label{c:allp}
Let $X$ be a Banach space in which the decoupling inequality \eqref{decoup}
holds for some $p\in (0,\infty)$, then for every $\Phi$ as defined in the introduction and every $(\F_n)_{n\geq 1}$-adapted
$X$-valued sequence $(f_n)_{n\geq 1}$ such that $d_n$ is $\F_{n-1}$-conditionally symmetric for all $n\geq 1$ one has:
\begin{align*}
\E \Phi ( f^*_n) \leq C_{X,r,p,q} \E \Phi (g_n), \ \ n\geq 1,
\end{align*}
where $g$ is a decoupled sum sequence of $f$ and $C_{X,r,p,q}$ is as given in \eqref{eq:Cpq}.
\end{corol}

\begin{remark}
From the proof of Theorem \ref{t:allpq} it follows that in order to check
whether a \mbox{(quasi-)}Banach space satisfies the decoupling inequality it suffices to
check whether the following weak estimate holds: for some $p\in (0,\infty)$
and some $b\in(0,1)$ there exists an $A=A(b,X,r,p)$ such that
\begin{align*}
\sup_{f\in \mathcal{D}_{\infty}}\sup_{0\leq k\leq l}\sup_{B\in \F_k, B\neq \varnothing}\P(\n ^{k}\!\!f^{l} \n \geq A\n T_p(^{k}\!\!f^{l})\n_{\infty}\,|\, B)\leq b.
\end{align*}
After all, if this holds for some $b\in (0,1)$, there will be a $p\in (0,\infty)$ such that $b\leq 2^{-\frac{2p}{r}-1}$. We then take $\Phi=x^p$ in Proposition \ref{p:allp2} (i.e.\ $q=p$) and obtain that \eqref{Tallp} holds for $f\in \mathcal{D}_{\infty}$ on a arbitrary yet fixed complete probability space $(\O,(\F_n)_{n\geq 1},\mathcal{A},\P)$. By the same arguments as in the proof of Theorem \ref{t:allpq} above we find that the decoupling inequality holds in $p$ for $X$, and thus, by Theorem \ref{t:allpq}, $X$ is a Banach space for which the decoupling inequality holds.
\end{remark}

\begin{corol}\label{cor:umd}
If $X$ is a \umd space, then the decoupling inequality holds.
\end{corol}
\begin{proof}
As explained in the introduction, if $X$ is a \textsc{umd} space then
\eqref{decoup-2} holds for all martingale difference sequences and for all
$p\in (1, \infty)$. Therefore, by Lemma \ref{l:condsym} and Theorem \ref{t:allpq} every \textsc{umd}
space satisfies the decoupling inequality.
\end{proof}

The lemma below implies that the decoupling property is a super-property: if $X$ is a quasi-Banach space satisfying the decoupling inequality and $Y$ is a quasi-Banach space that is finitely representable in $X$, then $Y$ satisfies the decoupling inequality and $D_p(Y)\leq D_p(X)$, $p\in (0,\infty)$. For the definition of finite representability we refer to \cite{AlbiacKalton}.

\begin{lemma}\label{lem:local}
A quasi-Banach space $X$ satisfies the decoupling inequality in $p\in(0,\infty)$ with constant $D_p(X)$ if and only if \eqref{decoup} in Definition \ref{def:dec} holds with constant $D_p(X)$ for every finitely-valued $X$-valued $(\F_n)_{n\geq1}$-adapted finite sequence $f=(f_k)_{k=1}^{n}$, for any probability space $(\O,\mathcal{A},(\F_{n})_{n\geq 1},\P)$. \par
\end{lemma}

\begin{proof}
Fix $p\in (0,\infty)$. It is clear from the definition that it suffices to consider finite sequences. Let $(\O,\mathcal{A},(\F_{n})_{n\geq 1},\P)$ be a probability space and let $(f_k)_{k=1}^{n}$ be a $X$-valued, $(\F_{k})_{k= 1}^n$-adapted $L^p$-sequence, and let $(g_k)_{k=1}^{n}$ be the decoupled sum sequence of $(f_k)_{k=1}^{n}$. By strong measurability we may assume that $(f_k)_{k=1}^{n}$ and $(g_k)_{k=1}^{n}$ take values in a separable subspace $X_0\subseteq X$. Let $(x_n)_{n\geq 1}$ be a dense subset of $X_0$ such that $x_1=0$.
For $m\in \N$ we define $\phi_m:X\rightarrow \R$ by
$$ \phi_m(x) = \min_{1\leq n\leq m}\{ \n x-x_n\n\,:\,\n x_n \n\leq \n x\n\}. $$
For $n,m\in \N$, $n\leq m$ define $E_{n,m}:=\{x\in X\,:\, \n x-x_n\n=\phi_m(x)\}.$ Define $\psi_m:X\rightarrow \{x_1,\ldots,x_m\}$ by
\begin{align*}
\psi_m(x)& = x_n; \qquad x\in E_{n,m}\setminus \bigcup_{j=1}^{n-1} E_{j,m}.
\end{align*}
Clearly, $\psi_m$ is $\mathcal{B}(X)$-measurable.
Moreover, for all $x\in X$ one has $\n \psi_m(x) - x\n\rightarrow 0$ as $m\rightarrow \infty$, and $\psi_m(x)\leq \|x\|$. Thus by the dominated convergence theorem we have $\psi_m(f_k)\rightarrow f_k$ and $\psi_m(g_k)\rightarrow g_k$ in $L^p(X)$, for all $k=1,\ldots,n$. By Lemma \ref{lem:preddcoupl}, $\psi_m(g_k)$ is the decoupled sum sequence for $\psi_m(f_k)$ for all $m\in \N$, so if \eqref{decoup} holds for the pairs $\psi_m(f_k)$ and $\psi_m(g_k)$ for all $m$ with some constant $D_p$, then it also holds for $(f_k)_{k=1}^{n}$ and $(g_k)_{k=1}^{n}$ with the same constant.
\end{proof}

\begin{corol}\label{cor:LqY}
Let $Y$ be a space for which the decoupling inequality \eqref{decoup} holds.
Let $(S,\Sigma,\mu)$ be a nonzero measure space and let $q\in (0,
\infty)$. Then $X=L^q(S;Y)$ satisfies the decoupling inequality. Moreover,
$D_p(L^p(S;Y)) = D_p(Y)$.
\end{corol}
\begin{proof}[Proof of Corollary \ref{cor:LqY}]
By Lemma \ref{lem:local} it suffices to consider finite sequences taking values in a finite subset of $L^p(S;Y)$. Thus without loss of generality we may assume that $(S,\Sigma,\mu)$ is $\sigma$-finite. Then the proof follows from Theorem \ref{t:allpq} by the same method as in \cite[Theorem 14]{CoxVeraar}, where $q=1$ has been considered.
\end{proof}

In particular, we have the following examples.
\begin{example}\label{ex:Lp}
Let $(S_i, \Sigma_i, \mu_i)$ be a measure space and let
$q_i\in (0, \infty)$ for $1\leq i\leq n$. Let $X = L^{q_1}(S_1; L^{q_2}(S_2;
\ldots L^{q_n}(S_n)))$, then the decoupling inequality holds for $X$. Note
these spaces are not \umd spaces if $q_i\leq 1$ for some $i$.
\end{example}

\begin{example}
Let $(S, \Sigma)$ be a measurable space. Let $X$ be the space of bounded $\sigma$-additive measures on $(S, \Sigma)$ equipped by the variation norm. Then $X$ is a Banach lattice where $\mu_1\leq \mu_2$ if $\mu_1(A) \leq\mu_2(A)$ for all $A\in \Sigma$. Moreover, $X$ is an abstract $L^1$-space and hence by \cite[Theorem 4.27]{AlBu}, the decoupling property holds for $X$.
\end{example}

A consequence of Corollary \ref{cor:LqY} is the following result for Hilbert
spaces $X$.
\begin{corol}\label{cor:uniformH}
Let $X$ be a Hilbert space. Then for every $p\in[1,\infty]$ and every adapted $X$-valued $f$ in $L^p(\O;X)$ one has:
\begin{align*}
\Vert f_n \Vert_p \leq D_{\R} \Vert g_n \Vert_p,
\end{align*}
for all $n\in\mathbb{N}$, where $g$ is a decoupled sum sequence of $f$ and $D_{\R}$ as in \eqref{uniformR}.
\end{corol}
Using this we prove a similar statement for estimates of type \eqref{decoup1q},
see inequality \eqref{uniformPhiH}.
Note that it has been proven that a Hilbert space $X$ satisfies the decoupling inequality in \cite[Corollary 6.4.3]{delaPG}, but it has not been proven that the constants $D_p(X)$ are uniformly bounded. It seems that the arguments of \cite{hitczenko:decoupineq}, \cite[Chapter 7]{delaPG} do {\em not} extend to the vector-valued situation and a different argument is used.
\begin{proof}
Let $X$ be a Hilbert space. Let $p\in [1,\infty)$ be given and let $f$ be an
adapted $X$-valued $L^p$-sequence. Because $f$ is strongly measurable we may
assume that $X$ is separable. As every separable Hilbert space is
isometrically isomorphic to a closed subspace of $\ell^2$ we may assume $f$ to
be an adapted $\ell^2$-valued $L^p$-sequence. It is known that $\ell^2$ embeds
isometrically in $L^p(0,1)$ for all $1\leq p < \infty$ (see \cite[Proposition
6.4.13]{AlbiacKalton}), let $J_p:\ell^2\rightarrow L^p(0,1)$ denote this
isometric embedding. Let $g$ be a decoupled sum sequence of $f$. Observe
that $J_p f$ is an adapted $L^p(0,1)$-valued $L^p$-sequence with decoupled sum
sequence $J_p g$. By equation \eqref{uniformR} in the introduction and
Corollary \ref{cor:LqY} it follows that, for all $n\geq 1$,
\begin{align*}
\Vert f_n \Vert_{L^p(\O;\ell^2)} = \Vert J_p f_n \Vert_{L^p(\O;L^p(0,1))} \leq
D_{\mathbb{R}} \Vert J_p g_n \Vert_{L^p(\O;L^p(0,1))} = D_{\mathbb{R}}\Vert
g_n \Vert_{L^p(\O;\ell^2)}.
\end{align*}
\end{proof}

\begin{remark}
We mention some direct consequences of Corollary \eqref{cor:uniformH}. Let $X$ be a Hilbert space.
\begin{enumerate}
\item Let $\Phi$ be fixed and $q$ as in \eqref{qPhi}. By Corollary \ref{cor:uniformH} we have $D_q(X)\leq D_{\R}$; using this and substituting $p=q$ in \eqref{decoupconst} we obtain:
\begin{align}\label{uniformPhiH}
\E \Phi(\n f \n) &\leq e^{q}2^{11+8q}[D_{\R}]^{q} \E \Phi(\n g\n),
\end{align}
for all $X$-valued sequences $f$. As mentioned in the introduction, this improves \cite[Corollary 6.4.3]{delaPG} where this estimate has been proven without giving a bound on the constant.
\item In \cite[Section 6]{hitczenko:decoupineq} it has been observed that if $D_p(X)$ is uniformly bounded in $p$ then using Taylor expansions one obtains estimates for $\E\Phi(\n f_n \n)$ even if $\Phi$ does not satisfy \eqref{qPhi}. This applies for example to the exponential function. I.e.\ by Corollary \ref{cor:uniformH} and Taylor expansions one has:
\begin{align*}
\E \exp(\n f_n \n) &\leq \E \exp( D_{\mathbb{R}} \n g_n \n),
\end{align*}
for all $X$-valued adapted sequences $f$. For the real case this estimate also follows for mean-zero sequences from a result in \cite[Section 6.2]{delaPG} (with constant $2$ instead of $D_{\mathbb{R}}$).
\end{enumerate}
\end{remark}

We conclude this section with some observations. In \cite[p.\
105]{Garling:RMTI} it has been proven that $c_0$ does not have the decoupling
property by proving that for any dimension $d$ one has $D_p(\ell^{\infty}_{(d)})\geq
4^{-1}K_{p,2}^{-1}[\tfrac{\log d}{\log 2}]^{\inv{2}}$ where $K_{p,2}$
is the optimal constant in the Kahane-Khintchine inequality.
We have the following upper estimate for $D_p(\ell^{\infty}_{(d)})$ for $p$ large:
\begin{corol}
Let $d\in\N$ and $p\geq \frac{\log d}{\log 2}$, then $D_p(\ell^{\infty}_{(d)})\leq 2 D_{\R}$.
\end{corol}
\begin{proof}
Recall from Corollary \ref{cor:LqY} that $D_p(\ell_p)=D_{\mathbb{R}}$ and
hence for any $\ell^{\infty}_{(d)}$ valued $L^p$-sequence $f$ with decoupled
sum sequence $g$ and any $n\in \N$, one has by H\"older's inequality:
\begin{align*}
(\E \n f_n \n^p_{\ell^{\infty}_{(d)}})^{\inv{p}}&\leq (\E \n f_n \n^p_{\ell^{p}_{(d)}})^{\inv{p}}
 \leq D_{\R} (\E \n g_n \n^p_{\ell^{p}_{(d)}})^{\inv{p}}\\
& \leq D_{\R} d^{\inv{p}}(\E \n g_n \n^p_{\ell^{\infty}_{(d)}})^{\inv{p}}
\leq 2D_{\R} (\E \n g_n \n^p_{\ell^{\infty}_{(d)}})^{\inv{p}}.
\end{align*}
\end{proof}

\begin{remark}\label{r:invnorms}
As in \cite{HitMS} the $L^p$-norms in the decoupling inequality \eqref{decoup} can be replaced by certain rearrangement invariant quasi-norms: Let $X$ be a quasi-Banach space satisfying the decoupling inequality and let $Y$ be a $(p,q)$-$K$-interpolation space for some $0< p,q < \infty$ on some complete probability space $(\Omega,\Sigma,\P)$. Then there exists a constant $D$ such that for all sequences $(f_n)_{n\geq 1}$ for which $\n f_n \n_X\in Y$ for all $n\geq 1$, with decoupled sum sequence $(g_n)_{n\geq 1}$ one has:
\begin{align*}
\big\n \n f_n \n_X \big\n_{Y} & \leq D\big\n \n g_n \n_X \big\n_{Y}, \quad \textrm{for all } n\geq 1.
\end{align*}
The proof of this statement is entirely analogous to \cite[Corollary 1.4]{HitMS}. Examples of $(p,q)$-$K$-interpolation spaces include all $(p,q)$-interpolation spaces with $1\leq p,q \leq \infty$ and the Lorentz spaces $L_{p,q}$ for $0<p,q< \infty$. Recall that a rearrangement invariant space $Y$ is an $(p,q)$-interpolation space if the Boyd indices $p_0,q_0$ satisfy $p<p_0$, $q>q_0$ \cite{Boyd69}.\par
\end{remark}

\begin{remark}
Let $X$ be a \textsc{umd} space and let $\mathcal{H}$ be the Hilbert transform on $L^p(\R;X)$ (or equivalently the periodic Hilbert transform on $L^p(0,2\pi;X)$). The estimate $\|\mathcal{H}\|_{\mathcal{L}(L^p(\R;X))}\leq \beta_p(X)^2$, where $\beta_p(X)$ is the \textsc{umd} constant of $X$, is the usual estimate in the literature (see \cite{Garling:BMandUMD,Bu3}). As the proofs in \cite{Garling:BMandUMD} and \cite{Bu3} involve only Paley-Walsh martingales, it follows that one actually has:
\begin{equation*}
\|\mathcal{H}\|_{\mathcal{L}(L^p(\R;X))}\leq C_p(X) D_p(X),
\end{equation*}
where $C_p(X)$ and $D_p(X)$ are as in \eqref{decoup-2}. Recall that $\max\{C_p(X),D_p(X)\}\leq \beta_p(X)$. Moreover, the behavior of $D_p(X)$ as $p\downarrow 1$ is better than $\beta_p(X)$. Indeed, according to Theorem \ref{t:allpq} one has $\sup_{p\in [1,2]} D_p(X)<\infty$, but $\beta_p(X)\to \infty$ as $p\downarrow 1$. Although we do not know whether $\sup_{p\in [2, \infty)}D_p(X)<\infty$, still a similar behavior occurs for the norm of $\mathcal{H}$ as $p\to \infty$. This follows from a duality argument. Indeed, recall that $X^*$ is a \textsc{umd} space again, and if $p\in [2, \infty)$, then with $1/p+1/p'=1$, we find:
\[\|\mathcal{H}\|_{\mathcal{L}(L^p(\R;X))}=\|\mathcal{H}^*\|_{\mathcal{L}(L^{p'}(\R;X^*))} \leq C_{p'}(X^*) D_{p'}(X^*).\]
Now $\sup_{p\in [2, \infty)}D_{p'}(X^*)<\infty$. Moreover, $C_{p'}(X^*)\leq \beta_{p'}(X^*) \leq \beta_p(X)$ by a duality argument.
\end{remark}

\section{Applications to stochastic integration}\label{s:bdg}
In this section let $X$ be a Banach space, $(\Omega,(\mathcal{F}_t)_{t\geq 0},P)$ a complete probability space and $H$ a real separable Hilbert space.

Recall from the introduction that it has been proven independently by both
Hitczenko \cite{hitczenko:notes} and McConnell \cite{mcconnell:decoup} that
$X$ is a \textsc{umd} space if and only if the two-sided decoupling inequality
holds for $1<p<\infty$. It has been shown in \cite{vanNeervenVeraarWeis} how this
two-sided decoupling inequality leads to a generalization of the
Burkholder-Davis-Gundy inequalities for $X$-valued stochastic integrals for
$1<p<\infty$ (inequality \eqref{bdg} below). A variant of the decoupling
inequality considering only conditionally symmetric sequences, presented in
\cite[Theorem 3']{hitczenko:notes}, will allow us to obtain continuous time
inequalities for \textsc{umd} spaces for $p\in
(0,\infty)$ (see Theorem \ref{t:stochint_bdg} below). The same technique can be applied to obtain one-sided
Burkholder-Davis-Gundy inequalities for spaces in which the decoupling
inequality holds.

Before formulating the generalized Burkholder-Davis-Gundy inequalities we
recall some theory on stochastic integration in Banach spaces as introduced in
\cite{vanNeervenVeraarWeis}.

An {\em $H$-cylindrical $(\mathcal{F}_t)_{t\geq 0}$-Brownian motion} is a mapping $$W_{H}:L^2(0,T;H)\rightarrow L^2(\Omega)$$ with the following properties:
\begin{enumerate}
\item for all $h\in L^2(0,T;H)$ the random variable $W_{H}(h)$ is Gaussian;
\item for all $h_1,h_2\in L^2(0,T;H)$ we have $\E W_{H}(h_1)W_{H}(h_2)=\langle h_1,h_2\rangle$;
\item for all $h\in H$ and all $t\in [0,T]$ we have that $W_H(\one_{[0,t]}\otimes h)$ is $\mathcal{F}_t$-measurable;
\item for all $h\in H$ and all $s,t\in [0,T]$, $s\leq t$ we have $W_H(\one_{[s,t]}\otimes h)$ is independent of $\mathcal{F}_s$.
\end{enumerate}
One easily checks that $W_{H}$ is linear and that for all $h_1,\ldots,h_n \in L^2(0,T;H)$ the random variables $W_{H}(h_1),\ldots,W_{H}(h_n)$ are jointly Gaussian. These random variables are independent if and only if $h_1,\ldots,h_n$ are orthogonal in $H$. With slight abuse of notation we will write $W_{H}(t)h:=W_H(\one_{[0,t]}\otimes h)$ for $t\in [0,T]$ and $h\in H$.\par
\begin{defn}
A process $\Psi:[0,\infty)\times \Omega \rightarrow \mathcal{L}(H,X)$ is called $H$-\emph{strongly
measurable} if for every $h\in H$ the process $\Psi h$ is strongly measurable.
The process is called $(\mathcal{F}_t)_{t\geq 0}$-\emph{adapted} if $\Psi h$ is $(\mathcal{F}_t)_{t\geq 0}$-adapted for each $h\in H$.
The process $\Psi$ is said to be {\em scalarly in $L^{0}(\Omega; L^2(0,T;H))$} if for all $x^*\in X^*$, $\Phi^*x^*\in L^{0}(\Omega; L^2(0,T;H))$.
\end{defn}

To build stochastic integrals of $\mathcal{L}(H,X)$-valued processes we start by
considering $(\mathcal{F}_t)_{t\geq 0}$-adapted finite rank step processes, i.e.\ processes of the form
\begin{align}\label{frsp}
\Psi(t,\omega)= \sum_{n=1}^{N}\one_{(t_{n-1},t_n]}(t) \sum_{m=1}^{M} h_m
\otimes \xi_{nm}(\omega),
\end{align}
where $0=t_0<t_1<...<t_N=T$, $\xi_{nm}\in L^0(\mathcal{F}_{t_{n-1}}; X)$ and
$(h_m)_{m\geq 1}$ is an orthonormal system in $H$. If $W_H$ is an
$H$-cylindrical $(\mathcal{F}_t)_{t\geq 0}$-Brownian motion, then
the stochastic integral process of $\Psi$ with respect to $W_H$ is given by:
\begin{align*}
\int_0^{t} \Psi \,dW_H &=\sum_{n=1}^{N} \sum_{m=1}^{M}
(W_H(t_n\minsym t)h_m-W_H(t_{n-1}\minsym t)h_m) \xi_{nm},
\end{align*}
for $t\in [0,T]$. Note that $t\mapsto \int_{0}^{t} \Psi \,dW_H$ is continuous almost surely.
For general Banach-space valued processes the stochastic integral is defined as
follows:
\begin{defn}\label{d:Lp_stochInt}
Let $W_H$ be an $H$-cylindrical $(\mathcal{F}_t)_{t\geq 0}$-Brownian motion. An
$H$-strongly measurable $(\mathcal{F}_t)_{t\geq 0}$-adapted process $\Psi:[0,T]\times \Omega \rightarrow
\mathcal{L}(H,X)$ is called \emph{stochastically integrable with respect to $W_H$} if
there exists a sequence of finite rank $(\mathcal{F}_t)_{t\geq 0}$-adapted step processes
$\Psi_n:[0,T]\times \Omega\rightarrow \mathcal{L}(H,X)$ such that:
\begin{enumerate}
\item for all $h\in H$ we have $\lim_{n \rightarrow \infty}\Psi_n h =  \Psi h$
in measure on $[0,T]\times \Omega$;
\item there exists a process $\zeta\in L^0(\Omega; C([0,T]; X))$ such that
\begin{align*}
\lim_{n\rightarrow \infty}\int_{0}^{\cdot} \Psi_n \,dW_H & = \zeta \qquad
\textrm{in } L^0(\Omega; C([0,T]; X)).
\end{align*}
\end{enumerate}
We define $\int_0^{\cdot} \Psi \,dW_H := \zeta$.
\end{defn}
The $\gamma$-radonifying operators defined below generalize the concept of Hilbert-Schmidt
operators and prove to be useful in the context of vector-valued stochastic
integration:
\begin{defn}
Let $X$ be a Banach space and $\mathcal{H}$ a Hilbert space. An operator $R\in
\mathcal{L}(\mathcal{H},X)$ is said to be a $\gamma$-radonifying if there exists an orthonormal
basis $(h_n)_{n\geq 1}$ of $\mathcal{H}$ such that if
$(\gamma_n)_{n\geq 1}$ is a sequence of independent Gaussian random
variables on $\Omega$, the Gaussian series $\sum_{n=1}^{N}\gamma_n R h_n$
converges in $L^2(\Omega,X)$. We define
\begin{align*}
\Vert R \Vert_{\gamma(\mathcal{H},X)} &:=\Big( \mathbb{E} \Big\Vert
\sum_{n=1}^{\infty} \gamma_n R h_n \Big\Vert^{2} \Big)^{\inv{2}}.
\end{align*}
(One checks that this norm is independent of the orthonormal basis.)
\end{defn}

Now we consider the case where $\mathcal{H}= L^2(0,T;H)$ and $H$ is a Hilbert space as before.
Let $\Psi:[0,T]\times\O\to \mathcal{L}(H,X)$ be $H$-strongly measurable and such that a.s.\
for all $x^*\in X^*$, $\Psi^* x^* \in L^2(0,T;H)$.
Then for almost all $\omega\in \O$ we can define a Pettis integral operator $R_{\Psi}(\omega)\in \mathcal{L}(\mathcal{H},X)$ as follows:
\begin{align*}
R_{\Psi}(\omega)f &:= \int_{0}^{T} \Psi(t,\omega)f(t)\, dt,
\end{align*}
for $f\in L^2([0,T];H)$.
In the following we simply identify $R_{\Psi}$ with $\Psi$, i.e.\ we write
$\Psi$ instead of $R_{\Psi}$ and set
\begin{align*}
\Vert \Psi(\cdot, \omega) \Vert_{\gamma(0,T;H,X)}&:= \Vert R_\Psi(\omega) \Vert_{\gamma(L^2(0,T;H),X)},
\end{align*}
whenever $R_\Psi(\omega)\in \gamma(L^2(0,T;H),X)$.

The following Theorem is an extension of some results presented in \cite{vanNeervenVeraarWeis}. We prove it by using the decoupling inequalities \eqref{umddecoup} and the inequality in Remark \eqref{r:geiss}. Note that we write $A \lesssim_t B$ if there exists a constant $C$ depending only on a parameter $t$, such that $A\leq C B$. Naturally $A \gtrsim_t B$ means $B \lesssim_t A$ and $A\eqsim_t B$ means
$A \lesssim_t B$ and $B \lesssim_t A$.
\begin{thm}\label{t:stochint_bdg}
Let $X$ be a Banach space and $H$ be a separable Hilbert space. Let $W_H$ be an
$H$-cylindrical $(\mathcal{F}_t)_{t\geq 0}$-Brownian motion. Let $\Psi:[0,T]\times \Omega \rightarrow \mathcal{L}(H,X)$ be an $H$-strongly
measurable and $(\mathcal{F}_t)_{t\geq 0}$-adapted process which is scalarly in $L^{0}(\Omega; L^2(0,T;H))$.
\begin{enumerate}
\item[(1)] If $X$ is a \textsc{umd} space, then the
following assertions are equivalent:
\begin{enumerate}
\item[(i)] $\Psi$ is stochastically integrable with respect to $W_H$;
\item[(ii)] $\Psi\in \gamma(0,T;H,X)$ a.s.
\end{enumerate}
Moreover, for $p\in (0,\infty)$ the following
continuous time Burkholder-Davis-Gundy inequalities hold:
\begin{align}\label{bdg}
\E \sup_{0\leq t \leq T}\Big\Vert \int_{0}^{t} \Psi \,dW_H \Big\Vert^p
&\eqsim_{p, X} \E \Vert \Psi \Vert^p_{\gamma(0,T;H,X)}.
\end{align}
\item[(2)] If $X$ satisfies the decoupling
inequality then (ii) $\Rightarrow$ (i) above still holds and for $p\in
(0,\infty)$ there exists a constant $\kappa_{p,X}$ such that:
\begin{align}\label{bdg_ineq}
\E \sup_{0\leq t \leq T}\Big\Vert \int_{0}^{t} \Psi \,dW_H \Big\Vert^p
&\leq \kappa_{p,X}^p \E \Vert \Psi \Vert^p_{\gamma(0,T;H,X)},
\end{align}
whenever the right-hand side is finite. Moreover, one can take $\kappa_{p,X}$ such that $\sup_{p\geq 1} \kappa_{p,X}/p<\infty$.
\end{enumerate}
\end{thm}

\begin{remark}
The constants in \eqref{bdg} and \eqref{bdg_ineq} are independent of $T$, and it is not difficult to see that one can also take $T=\infty$. Since every \umd space satisfies the decoupling inequality (see Corollary \ref{cor:umd}), the estimate \eqref{bdg_ineq} holds for \umd spaces $X$ with the same behavior of the constant $\kappa_{p,X}$. Already for $X = L^q$ with $q\neq 2$, it is an open problem whether the optimal constant $\kappa_{p,X}$ satisfies $\sup_{p\geq 1} \kappa_{p,X}/\sqrt{p}<\infty$. For Hilbert spaces this is indeed the case.
\end{remark}

For Theorem \ref{t:stochint_bdg} (1) the equivalence of (i) and (ii) and the
Burkholder-Davis-Gundy inequalities for $p>1$ are given in \cite[Theorems 5.9
and 5.12]{vanNeervenVeraarWeis}. Thus for the proof it remains to prove the Burkholder-Davis-Gundy inequalities for the case
$0<p\leq 1$. From \cite[Theorem 3']{hitczenko:notes} (see also
\cite[Proposition 2]{CoxVeraar}) we know that if $X$ is a \textsc{umd} space
then one has:
\begin{align}\label{umddecoup}
\Vert f^* \Vert_p \eqsim_{p,X} \Vert g \Vert_p,
\end{align}
for all $p\in (0,\infty)$, for $(\F_n)_{n\geq 1}$-adapted $X$-valued $L^p$-sequences $(f_n)_{n\geq 1}$ on some complete probability space such that $f_n-f_{n-1}$ is $\F_{n-1}$-conditionally symmetric for all $n\geq 1$. The idea of the proof of Theorem
\ref{t:stochint_bdg} is taken from \cite[Lemma 3.5]{vanNeervenVeraarWeis}), an alternative approach would be to use the extrapolation results in \cite{lenglart}.

\begin{remark} \
\begin{enumerate}
\item Let $(\Omega, \mathcal{F},\P,(\mathcal{F}_i)_{i=0}^{n})$ be a probability space endowed with a filtration. Let $n\in \N$ and let $g_1,\ldots,g_n$ be independent standard Gaussian random variables on $(\Omega, \mathcal{F},\P,(\mathcal{F}_i)_{i=1}^{n})$ such that $g_i$ is $\mathcal{F}_{i}$-measurable and independent of $\mathcal{F}_{i-1}$. Let $(\tilde{g}_1,\ldots,\tilde{g}_n)$ be a copy of $(g_1,\ldots,g_n)$ independent of $(\Omega, \mathcal{F},\P,(\mathcal{F}_i)_{i=0}^{n})$. From the proof of Theorem \ref{t:stochint_bdg} it follows that in order to prove that the Burkholder-Davis-Gundy inequality (5.3) is satisfied for processes in a Banach space $X$, for some $p\in (0,\infty)$, it suffices to prove that there exists a constant $c_p$ such that
\begin{align}\label{bdg_core}
\Big\n \sup_{1\leq j\leq n} \Big\n \sum_{i=1}^{j} g_i v_{i-1} \Big\n \Big\n_{L^p(\Omega)}&\leq c_p \Big\n \sum_{i=1}^{n} \tilde{g}_i v_{i-1} \Big\n_{L^p(\Omega\times\tilde{\Omega},X)}
\end{align}
for all $(v_i)_{i=0}^{n-1}$ an $(\mathcal{F}_{i})_{i=0}^{n-1}$-adapted sequence of $X$-valued simple random variables and all $n\in \N$. For this it is sufficient that the decoupling inequality holds for $p$, but we do not know whether it is necessary. Similarly, in order to prove \eqref{bdg} it suffices to prove that one has a two-sided estimate in \eqref{bdg_core}. For this it is known that it is necessary and sufficient that $X$ is a \textsc{umd} Banach space (see \cite{Garling:BMandUMD}).
\item By studying the proof of \cite[Theorem 3']{hitczenko:notes} one may conclude
that if \eqref{umddecoup} holds for some $p\in (1,\infty)$, it holds for all
$p\in(1,\infty)$. As a result, and by considering Paley-Walsh martingales,
one can also prove that if \eqref{umddecoup} holds
in a Banach space $X$ for some $p\in (1,\infty)$, for all $X$-valued $L^p$-sequences $f$ with conditionally symmetric increments, then $X$ is a
\textsc{umd} space.
\item If for a Banach space $X$ there exists a
$p\in(1,\infty)$ such that \eqref{bdg} holds for all stochastically integrable
processes $\Phi$ then $X$ is a \textsc{umd} space (see \cite{Garling:BMandUMD}).
\item Suppose the filtration $(\F_t)_{t\geq 0}$ in Theorem \ref{t:stochint_bdg} has the form $$\F_t = \sigma(W_H(s) h: s\leq t, h\in H)$$ for each $t\in [0,\infty)$. In this case \eqref{bdg} for some $p\in (0,\infty)$ can be derived from \eqref{umddecoup} for Paley--Walsh martingales for that $p$. Similarly, \eqref{bdg_ineq} for some $p\in(0,\infty)$ can be derived from the corresponding one-sided estimate in \eqref{umddecoup} for Paley--Walsh martingales for that $p$. This follows from a central limit theorem argument as in \cite[Theorem 3.1]{GeMSS}. Conversely, \eqref{bdg_ineq} implies the one-sided estimate in \eqref{umddecoup} for Paley--Walsh martingales (see \cite{veraar:umdBM}).
\end{enumerate}
\end{remark}

\begin{proof}
(1): \ Let $p\in (0,\infty)$ be fixed and let $\Psi$ be
a finite-rank step process of the form \eqref{frsp} with $\xi_{nm}\in L^{\infty}(\mathcal{F}_{t_{n-1}},X)$, $W_H$ an $H$-cylindrical $(\mathcal{F}_t)_{t\geq 0}$-Brownian motion and let $\widetilde{W}_H$ be a copy of $W_H$ that is independent of $\F_{\infty} =
\sigma(\bigcup_{t\geq 0} \F_t)$. Then
\begin{align*}
\left(\sum_{m=1}^{M}  (W_H(t_n)h_m-W_H(t_{n-1})h_m) \xi_{nm} \right)_{n=1}^{N}
\end{align*}
is a $\F_{\infty}$-conditionally symmetric sequence and a
$\F_{\infty}$-decoupled version is given by $\left(\sum_{m=1}^{M}
(\widetilde{W}_H(t_n)h_m-\widetilde{W}_H(t_{n-1})h_m) \xi_{nm} \right)_{n=1}^{N}$.
One has:
\begin{align*}
& \E \sup_{1\leq j \leq N} \Big\Vert \int_0^{t_j} \Psi \,dW_H \Big\Vert^p\\
& \qquad \qquad =\E \sup_{1\leq j \leq N} \Big\Vert \sum_{n=1}^{j} \sum_{m=1}^{M} (W_H(t_n)h_m-W_H(t_{n-1})h_m) \xi_{nm}\Big\Vert^p\\
& \qquad \qquad \stackrel{(i)}{\eqsim}_{p,X} \E \Big\Vert \sum_{n=1}^{N} \sum_{m=1}^{M}  (\widetilde{W}_H(t_n)h_m-\widetilde{W}_H(t_{n-1})h_m) \xi_{nm}\Big\Vert^p \\
& \qquad \qquad \stackrel{(ii)}{\eqsim}_{p,X} \Big(\E \Big\Vert \sum_{n=1}^{N} \sum_{m=1}^{M}  (\widetilde{W}_H(t_n)h_m-\widetilde{W}_H(t_{n-1})h_m) \xi_{nm}\Big\Vert^2\Big)^{\frac{p}{2}} \\
& \qquad \qquad  = \E \Vert \Psi \Vert^p_{\gamma(0,T;H, X)},
\end{align*}
where equation (i) follows from equation \eqref{umddecoup} and
equation (ii) follows by the Kahane-Khintchine inequality (see \cite[Section 1.3]{delaPG}).

For $n\in\mathbb{N}$ let $D_n$ be the $n^{\textrm{th}}$ dyadic partition of
$[0,T]$, i.e.\ $D_n:=\{\tfrac{k}{2^n}:k=0,1,2,\ldots \}\cap [0,T]$ and define
$\widetilde{D}_n:=D_n\cup \{t_1,\ldots,t_N\}$. Then by the above one has:
\begin{align*}
\E \sup_{t\in \widetilde{D}_n} \Big\Vert \int_0^{t} \Psi \,dW_H \Big\Vert^p & \eqsim_{p,X}  \E \Vert \Psi
\Vert^p_{\gamma(0,T;H, X)}, \ \
n\in\mathbb{N}.
\end{align*}
By the monotone convergence theorem and path continuity of the integral process one has:
\begin{align*}
\E \sup_{t\in [0,T]} \Big\Vert \int_0^{t} \Psi \,dW_H \Big\Vert^p &=
\lim_{n\rightarrow \infty} \E \sup_{t\in \widetilde{D}_n} \Big\Vert \int_0^{t}
\Psi \,dW_H \Big\Vert^p \eqsim_{p,X}  \E \Vert \Psi \Vert^p_{\gamma(0,T;H, X)}.
\end{align*}
Hence equation \eqref{bdg} holds for finite-rank step processes.

Now let $\Psi$ be any stochastically integrable process. Suppose
\begin{align*}
\E \sup_{0\leq t \leq T}\Big\Vert \int_{0}^{t}\Psi \,dW_H\Big\Vert^p<\infty,
\end{align*}
then by an approximation argument as in the proof of \cite[Theorem 5.12]{vanNeervenVeraarWeis}) one has:
\begin{align*}
\E \Vert \Psi \Vert^p_{\gamma (0,T;H,X)} & \lesssim_{p,X} \E \sup_{t\in[0,T]}\Big\Vert \int_{0}^{t} \Psi \,dW_H
\Big\Vert^p.
\end{align*}
Hence it suffices to prove \eqref{bdg} under the assumption that
$\E \Vert \Psi \Vert^p_{\gamma(0,T;H,X)}<\infty$.
By a straightforward adaptation of the proof of \cite[Proposition 2.12]{vanNeervenVeraarWeis} one
can prove that there exists a sequence of finite-rank $(\mathcal{F}_t)_{t\geq 0}$-adapted step processes
$(\Psi_n)_{n\geq 1}$ that converges to $\Psi$ in
$L^p(\Omega,\gamma(0,T;H,X))$. Hence in particular for all $x^*\in X^*$, $\Psi^*_n x^*\to \Psi^* x^*$ in $L^0(\O;L^2(0,T;H))$. Because \eqref{bdg} holds for finite-rank
$(\mathcal{F}_t)_{t\geq 0}$-adapted step processes the sequence
$\left( \int_{0}^{\cdot}\Psi_n
\,dW_H \right)_{n\geq 1}$
is a Cauchy sequence in $L^p(\Omega, C([0,T];X))$. In particular
$(\Psi_n)_{n\geq 1}$ approximates the stochastic integral of $\Psi$ in the
sense of Definition \ref{d:Lp_stochInt} and
\begin{align*}
\E\sup_{0\leq t \leq T} \Big\Vert\int_{0}^{t} \Psi \,dW_H(t) \Big\Vert^p &= \lim_{n\rightarrow \infty}\E\sup_{0\leq t
\leq T} \Big\Vert\int_{0}^{t} \Psi_n \,dW_H(t) \Big\Vert^p\\
& \eqsim_{p,X} \lim_{n\rightarrow \infty}\E\Vert \Psi_n \Vert^p_{\gamma(0,T;H,X)} = \E\Vert \Psi \Vert^p_{\gamma(0,T;H,X)}.
\end{align*}

(2): \ Suppose $X$ satisfies the decoupling inequality. In this case the proof for finite-rank $(\mathcal{F}_t)_{t\geq 0}$-adapted step processes given above can be repeated using the inequality in Corollary \ref{c:allp} instead of equation \ref{umddecoup}. To prove \eqref{bdg_ineq}
for arbitrary processes we repeat the argument in (1) concerning the case
that
$\E \Vert \Psi \Vert^p_{\gamma(0,T;H,X)}<\infty$.\par
However, in order to obtain the estimate $\sup_{p\geq 1}\kappa_{X,p}/p<\infty$ we use inequality \eqref{eq:geissest} in the proof of Theorem \ref{t:allpq} in the following manner: when $p\in [1,\infty)$ we have
\begin{align*}
&\E \sup_{1\leq j \leq N} \Big\Vert \int_0^{t_j} \Psi \,dW_H \Big\Vert^p\\
& \quad    =\E \sup_{1\leq j \leq N} \Big\Vert \sum_{n=1}^{j} \sum_{m=1}^{M} (W_H(t_n)h_m-W_H(t_{n-1})h_m) \xi_{nm}\Big\Vert^p\\
& \quad \leq c_{X,2} p \, \Big\|\E\Big(\Big\Vert \sum_{n=1}^{j} \sum_{m=1}^{M}  (\widetilde{W}_H(t_n)h_m-\widetilde{W}_H(t_{n-1})h_m) \xi_{nm}\Big\Vert^2\, \Big|\, \F_{\infty} \Big)^{1/2}\Big\|_{p}^p
\\ & \quad  = c_{X,2} p \, \E \Vert \Psi \Vert^p_{\gamma(0,T;H, X)}.
\end{align*}
\end{proof}

If $X$ has type $2$, then by \cite{vanNW05} one has that $L^2(0,T;\gamma(H,X))\hookrightarrow
\gamma(0,T;X)$, hence Theorem \ref{t:stochint_bdg} (2) implies
inequality \eqref{bdg_ineq_L2} below (as was already observed in
\cite{vanNeervenVeraarWeis}). This inequality has been proven for $p\in (1,\infty)$
in \cite{Brz97}, \cite{Brz03} using different techniques.

\begin{corol}\label{cor:bdgL2}
If $X$ is a Banach space satisfying the decoupling inequality (e.g.\ a
\textsc{umd} space) and $X$ has type 2 then for each $p\in (0,\infty)$ there is a constant $\mathcal{C}_{p,X}$ such that one has:
\begin{align}\label{bdg_ineq_L2}
\E \sup_{0\leq t \leq T}\Big\Vert \int_{0}^{t} \Psi \,dW_H \Big\Vert^p &\leq \mathcal{C}_{p,X}
\E \Vert \Psi \Vert^p_{L^2(0,T;\gamma(H,X))},
\end{align}
whenever the right-hand side is finite.
\end{corol}
As in Theorem \ref{t:stochint_bdg} one again has $\sup_{p\geq 1}\mathcal{C}_{p,X}/p<\infty$ if $\mathcal{C}_{p,X}$ is the optimal constant in \eqref{bdg_ineq_L2}. However, in \cite{Seidlernew} it has been recently proved that one has $\sup_{p\geq 1}\mathcal{C}_{p,X}/\sqrt{p}<\infty$ in \eqref{bdg_ineq_L2}.\par

\end{document}